\documentclass[a4paper,10pt]{amsart}

\usepackage{amsmath,amsfonts,amssymb}
\usepackage[all]{xy}
\usepackage[english]{babel}
\usepackage[utf8x]{inputenc}
\usepackage{graphicx}
\usepackage{multirow}
\usepackage{array}
\usepackage{booktabs}
\usepackage{ifthen}
\usepackage{fullpage}
\usepackage{cancel}
\usepackage{paralist}
\usepackage{microtype}
\usepackage{lscape}

\newtheorem{prop}{Proposition}[section]
\newtheorem{thm}[prop]{Theorem}
\newtheorem{cor}[prop]{Corollary}
\newtheorem{lemma}[prop]{Lemma}

\theoremstyle{definition}
\newtheorem{defi}[prop]{Definition}
\newtheorem{rem}[prop]{Remark}
\newtheorem{ex}[prop]{Example}

\newtheorem{notation}[prop]{Notation}

\newcommand{\N}{\mathbb{N}}
\newcommand{\Z}{\mathbb{Z}}

\newcommand{\R}{\mathbb{R}}
\newcommand{\C}{\mathbb{C}}
\newcommand{\K}{\mathbb{K}}
\newcommand{\st}{\;:\;}

\newcommand{\sspace}{\cdot}
\newcommand{\ssspace}{\cdot\cdot}

\newcommand{\scalar}[2]{\left\langle #1 \,\middle|\, #2 \right\rangle}

\newcommand{\duale}[1] {{{#1}^{*}}}
\newcommand{\g}{\mathfrak{g}}

\newcommand{\paragrafoo}[1]{\smallskip {\bfseries \itshape [#1]}.}

\DeclareMathOperator{\im}{i}
\DeclareMathOperator{\imm}{im}
\DeclareMathOperator{\de}{d}
\DeclareMathOperator{\id}{id}
\DeclareMathOperator{\End}{End}
\DeclareMathOperator{\Hom}{Hom}
\DeclareMathOperator{\Tot}{Tot}
\DeclareMathOperator{\Doub}{Doub}
\DeclareMathOperator{\Cliff}{Cliff}

\newcommand{\del}{\partial}
\newcommand{\delbar}{\overline{\del}}

\newcommand{\GCD}[2]{\mathrm{GCD}\left(#1, \, #2\right)}

\title{Inequalities {\itshape à la} Fr\"olicher and cohomological decompositions}

\author{Daniele Angella}
\address[Daniele Angella]{Dipartimento di Matematica\\
Universit\`{a} di Pisa\\
Largo Bruno Pontecorvo 5, 56127\\
Pisa, Italy}
\email{angella@mail.dm.unipi.it}

\curraddr{Istituto Nazionale di Alta Matematica\\
Dipartimento di Matematica e Informatica\\
Universit\`{a} di Parma\\
Parco Area delle Scienze 53/A, 43124\\
Parma, Italy}
\email{daniele.angella@math.unipr.it}

\author{Adriano Tomassini}
\address[Adriano Tomassini]{Dipartimento Di Matematica e Informatica\\
Universit\`{a} di Parma\\
Parco Area delle Scienze 53/A, 43124\\
Parma, Italy}
\email{adriano.tomassini@unipr.it}

\thanks{This work was supported by the Project PRIN ``Varietà reali e complesse: geometria, topologia e analisi armonica'', by the Project FIRB ``Geometria Differenziale e Teoria Geometrica delle Funzioni'', and by GNSAGA of INdAM}

\begin{document}

\vspace{-2cm}
\begin{minipage}[l]{10cm}
{\sffamily
  D. Angella, A. Tomassini, Inequalities {\itshape à la} Fr\"olicher and cohomological decompositions,
  to appear in {\em J. Noncommut. Geom.}.
}
\end{minipage}
\vspace{2cm}

\begin{abstract}
 We study Bott-Chern and Aeppli cohomologies of a vector space endowed with two anti-commuting endomorphisms whose square is zero. In particular, we prove an inequality {\itshape à la} Fr\"olicher relating the dimensions of the Bott-Chern and Aeppli cohomologies to the dimensions of the Dolbeault cohomologies. We prove that the equality in such an inequality {\itshape à la} Fr\"olicher characterizes the validity of the so-called cohomological property of satisfying the $\del\delbar$-Lemma. As an application, we study cohomological properties of compact either complex, or symplectic, or, more in general, generalized-complex manifolds.
\end{abstract}

\maketitle

\section*{Introduction}

Given a compact complex manifold $X$, the \emph{Bott-Chern cohomology}, $H^{\bullet,\bullet}_{BC}(X)$, \cite{bott-chern}, and the \emph{Aeppli cohomology}, $H^{\bullet,\bullet}_{A}(X)$, \cite{aeppli}, provide useful invariants, and have been studied by several authors in different contexts, see, e.g., \cite{aeppli, bott-chern, bigolin, deligne-griffiths-morgan-sullivan, varouchas, alessandrini-bassanelli, schweitzer, kooistra, bismut, tseng-yau-3, angella-1, angella-tomassini-3}. In the case of compact K\"ahler manifolds, or, more in general, of compact complex manifolds satisfying the \emph{$\del\delbar$-Lemma}, the Bott-Chern and the Aeppli cohomology groups are naturally isomorphic to the Dolbeault cohomology groups.
The $\del\delbar$-Lemma for compact complex manifolds has been studied by P. Deligne, Ph.~A. Griffiths, J. Morgan, and D.~P. Sullivan in \cite{deligne-griffiths-morgan-sullivan}, where it is proven that the validity of the $\del\delbar$-Lemma on a compact complex manifold $X$ yields the formality of the differential graded algebra
$\left( \wedge^\bullet X \otimes_\R \C ,\, \de \right)$, \cite[Main Theorem]{deligne-griffiths-morgan-sullivan}; in particular, a topological obstruction to the existence of K\"ahler structures on compact differentiable manifolds follows, \cite[Lemma 5.11]{deligne-griffiths-morgan-sullivan}.
Furthermore, they showed that any compact manifold admitting a proper modification from a K\"ahler manifold (namely, a manifold in class $\mathcal{C}$ of Fujiki, \cite{fujiki}) satisfies the $\del\delbar$-Lemma, \cite[Corollary 5.23]{deligne-griffiths-morgan-sullivan}.
An adapted version of the $\del\delbar$-Lemma for differential graded Lie algebras has been considered also in \cite{goldman-millson} by W.~M. Goldman and J.~J. Millson, where they used a ``principle of two types'', see \cite[Proposition 7.3(ii)]{goldman-millson}, as a key tool to prove formality of certain differential graded Lie algebras in the context of deformation theory, \cite[Corollary page 84]{goldman-millson}. An algebraic approach to the $\del\delbar$-Lemma has been developed also by Y.~I. Manin in \cite{manin} in the context of differential Gerstenhaber-Batalin-Vilkovisky algebras, in order to study Frobenius manifolds arising by means of solutions of Maurer-Cartan type equations.
A generalized complex version of the $\del\delbar$-Lemma has been introduced and studied by G.~R. Cavalcanti in \cite{cavalcanti, cavalcanti-jgp}.

\medskip

Since Bott-Chern and Aeppli cohomologies on compact K\"ahler manifolds coincide with Dolbeault cohomology, in \cite{angella-tomassini-3}, we were concerned in studying Bott-Chern cohomology of compact complex (possibly non-K\"ahler) manifolds $X$, showing the following \emph{inequality \itshape{à la} Fr\"olicher}, which relates the dimensions of the Bott-Chern and Aeppli cohomologies to the Betti numbers, \cite[Theorem A]{angella-tomassini-3}:
$$ \text{for any }k\in\Z\;, \qquad \sum_{p+q=k} \left( \dim_\C H^{p,q}_{BC}(X) + \dim_\C H^{p,q}_{BC}(X) \right) \;\geq\; 2\,\dim_\C H^k_{dR}(X;\C) \;; $$
furthermore, the authors showed that the equality in the above inequality holds for every $k\in\Z$ if and only if $X$ satisfies the $\del\delbar$-Lemma, \cite[Theorem B]{angella-tomassini-3}.

It turns out that such results depend actually on the structure of double complex of $\left(\wedge^{\bullet,\bullet}X ,\, \del ,\, \delbar\right)$. In this paper, we are concerned in a generalization of the inequality {\itshape à la} Fr\"olicher in a more algebraic framework, so as to highlight the algebraic aspects. As an application, we recover the above results on the cohomology of compact complex manifolds, and we get results on the cohomology of compact symplectic manifolds and compact generalized complex manifolds: more precisely, characterizations of compact symplectic manifolds satisfying the Hard Lefschetz Condition and of compact generalized complex manifolds satisfying the $\de\de^{\mathcal{J}}$-Lemma are provided.

\medskip

More precisely, consider a double complex $\left(B^{\bullet,\bullet}, \, \del,\, \delbar\right)$ of $\K$-vector spaces (namely, a $\Z^2$-graded $\K$-vector space $B^{\bullet,\bullet}$ endowed with $\del\in\End^{1,0}(B^{\bullet,\bullet})$ and $\delbar\in\End^{0,1}(B^{\bullet,\bullet})$ such that $\del^2=\delbar^2=\del\delbar+\delbar\del=0$). Several cohomologies can be studied: other than the \emph{Dolbeault 
cohomologies}
$$
  H^{\bullet,\bullet}_{\left( \del ; \del \right)}\left(B^{\bullet,\bullet}\right) \;:=\; \frac{\ker\del}{\imm\del}
  \qquad \text{ and } \qquad
  H^{\bullet,\bullet}_{\left( \delbar ; \delbar \right)}\left(B^{\bullet,\bullet}\right) \;:=\; \frac{\ker\delbar}{\imm\delbar} \;,
$$
and than the cohomology of the associated total complex, $\left(\Tot^\bullet B^{\bullet,\bullet} := \bigoplus_{p+q=\bullet} B^{p,q},\, \de:=\del+\delbar\right)$,
$$ H^\bullet_{\left( \de ; \de \right)}\left(\Tot^\bullet B^{\bullet,\bullet}\right) \;:=\; \frac{\ker\de}{\imm\de} \;,$$
one can consider also the \emph{Bott-Chern cohomology} and the \emph{Aeppli cohomology}, that is,
$$
  H^{\bullet,\bullet}_{\left( \del , \delbar ; \del\delbar \right)}\left(B^{\bullet,\bullet}\right) \;:=\; \frac{\ker\del\cap\ker\delbar}{\imm\del\delbar}
  \qquad \text{ and } \qquad
  H^{\bullet,\bullet}_{\left( \del\delbar ; \del , \delbar \right)}\left(B^{\bullet,\bullet}\right) \;:=\; \frac{\ker\del\delbar}{\imm\del+\imm\delbar} \;.
$$
The identity induces natural morphisms of (possibly $\Z$-graded, possibly $\Z^2$-graded) $\K$-vector spaces:
$$
\xymatrix{
 & H^{\bullet,\bullet}_{\left( \del , \delbar ; \del\delbar \right)}\left(B^{\bullet,\bullet}\right) \ar[ld] \ar[rd] \ar[d] & \\
 H^{\bullet,\bullet}_{\left( \del ; \del \right)}\left(B^{\bullet,\bullet}\right) \ar[rd] & H^\bullet_{\left( \de ; \de \right)}\left(\Tot^\bullet B^{\bullet,\bullet}\right) \ar[d] & H^{\bullet,\bullet}_{\left( \delbar ; \delbar \right)}\left(B^{\bullet,\bullet}\right) \ar[ld] \\
& H^{\bullet,\bullet}_{\left( \del\delbar ; \del , \delbar \right)}\left(B^{\bullet,\bullet}\right) &
}
$$
In general, the above maps are neither injective nor surjective; actually, the map $H^{\bullet,\bullet}_{\left( \del , \delbar ; \del\delbar \right)}\left(B^{\bullet,\bullet}\right) \to H^{\bullet,\bullet}_{\left( \del\delbar ; \del , \delbar \right)}\left(B^{\bullet,\bullet}\right)$ being injective is equivalent to all the above maps being isomorphisms, \cite[Lemma 5.15, Remark 5.16, 5.21]{deligne-griffiths-morgan-sullivan}. In such a case, one says that $\left( B^{\bullet,\bullet},\, \del,\, \delbar\right)$ \emph{satisfies the $\del\delbar$-Lemma}.

By considering the spectral sequence associated to the structure of double complex of $\left(B^{\bullet,\bullet},\, \del,\, \delbar\right)$, one gets the \emph{Fr\"olicher inequality}, \cite[Theorem 2]{frolicher},
$$ \min \left\{ \dim_\K \Tot^\bullet H^{\bullet,\bullet}_{\left( \del ; \del \right)}\left(B^{\bullet,\bullet}\right) ,\, \dim_\K \Tot^\bullet H^{\bullet,\bullet}_{\left( \delbar ; \delbar \right)}\left(B^{\bullet,\bullet}\right) \right\} \;\geq\; \dim_\K H^\bullet_{\left( \de ; \de \right)}\left(\Tot^\bullet B^{\bullet,\bullet}\right) \;.$$

We prove an \emph{inequality {\itshape à la} Fr\"olicher} also for the Bott-Chern and Aeppli cohomologies. More precisely, we prove the following result.

\smallskip
\noindent {\bfseries Theorem 1 (see Theorem \ref{thm:disug-frol} and Corollary \ref{cor:frolicher-like-double-complexes}).\ }
{\itshape
Let $A^\bullet$ be a $\Z$-graded $\mathbb{K}$-vector space endowed with two endomorphisms $\delta_1 \in \End^{\hat\delta_1}\left(A^\bullet\right)$ and $\delta_2 \in \End^{\hat\delta_2}\left(A^\bullet\right)$ such that $\delta_1^2 = \delta_2^2 = \delta_1\delta_2+\delta_2\delta_1 = 0$.
Suppose that
$$ \dim_\K H^\bullet_{\left( \delta_1 ; \delta_1 \right)}\left(A^\bullet\right) < +\infty  \qquad \text{ and } \qquad  \dim_\K H^\bullet_{\left( \delta_2 ; \delta_2 \right)}\left(A^\bullet\right) \;<\; +\infty \;.$$
Then
$$
\dim_\K H^\bullet_{\left( \delta_1 , \delta_2 ; \delta_1\delta_2 \right)}\left(A^\bullet\right) + \dim_\K H^\bullet_{\left( \delta_1\delta_2 ; \delta_1 , \delta_2 \right)}\left(A^\bullet\right) \;\geq\; \dim_\K H^\bullet_{\left( \delta_1 ; \delta_1 \right)}\left(A^\bullet\right) + \dim_\K H^\bullet_{\left( \delta_2 ; \delta_2 \right)}\left(A^\bullet\right) \;. 
$$

In particular, given a bounded double complex $\left( B^{\bullet,\bullet},\, \del,\, \delbar \right)$, and supposed that
$$ \dim_\K \Tot^\bullet H^{\bullet,\bullet}_{\left( \delta_1 ; \delta_1 \right)}\left(B^{\bullet,\bullet}\right) < +\infty  \qquad \text{ and } \qquad  \dim_\K \Tot^\bullet H^{\bullet,\bullet}_{\left( \delta_2 ; \delta_2 \right)}\left(B^{\bullet,\bullet}\right) \;<\; +\infty \;,$$
then, for $\pm\in\{+,-\}$,
$$
 \dim_\K \Tot^\bullet H^{\bullet,\bullet}_{\left( \delta_1 , \delta_2 ; \delta_1\delta_2 \right)}\left(B^{\bullet,\bullet}\right) + \dim_\K \Tot^\bullet H^{\bullet,\bullet}_{\left( \delta_1\delta_2 ; \delta_1 , \delta_2 \right)}\left(B^{\bullet,\bullet}\right) \;\geq\; 2\, \dim_\K H^\bullet_{\left( \delta_1 \pm \delta_2 ; \delta_1 \pm \delta_2 \right)}\left(\Tot^\bullet B^{\bullet,\bullet}\right) \;. 
$$
}
\smallskip

Furthermore, we provide a characterization of the equality in the above inequality {\itshape à la} Fr\"olicher in terms of the validity of the $\delta_1\delta_2$-Lemma.

\smallskip
\noindent {\bfseries Theorem 2 (see Theorem \ref{thm:caratt-deldelbar-lemma-double}).\ }
{\itshape
 Let $\left( B^{\bullet,\bullet},\, \delta_1,\, \delta_2 \right)$ be a bounded double complex.
 Suppose that
 $$ \dim_\K H^{\bullet,\bullet}_{\left( \delta_1 ; \delta_1 \right)}\left(B^{\bullet,\bullet}\right) < +\infty  \qquad \text{ and } \qquad  \dim_\K H^{\bullet,\bullet}_{\left( \delta_2 ; \delta_2 \right)}\left(B^{\bullet,\bullet}\right) \;<\; +\infty \;.$$
 The following conditions are equivalent:
 \begin{enumerate}
  \item[{\itshape (\ref{item:caratt-bi-1})}] $B^{\bullet,\bullet}$ satisfies the $\delta_1\delta_2$-Lemma;
  \item[{\itshape (\ref{item:caratt-bi-2})}] the equality
        \begin{eqnarray*}
        \lefteqn{ \dim_\K \Tot^{\bullet} H^{\bullet,\bullet}_{\left( \delta_1 , \delta_2 ; \delta_1\delta_2 \right)}\left(B^{\bullet,\bullet}\right) + \dim_\K \Tot^{\bullet} H^{\bullet,\bullet}_{\left( \delta_1\delta_2 ; \delta_1 , \delta_2 \right)}\left(B^{\bullet,\bullet}\right) } \\[5pt]
        &=& 2\, \dim_\K H^{\bullet}_{\left( \delta_1+\delta_2 ; \delta_1+\delta_2 \right)}\left(\Tot^\bullet B^{\bullet,\bullet}\right) \;.
        \end{eqnarray*}
        holds.
 \end{enumerate}

}
\smallskip

\medskip

Given a compact complex manifold $X$, one can apply Corollary \ref{cor:frolicher-like-double-complexes} and Theorem \ref{thm:caratt-deldelbar-lemma-double} to the double complex $\left(\wedge^{\bullet,\bullet}X ,\, \del ,\, \delbar\right)$. More precisely, one recovers \cite[Theorem A]{angella-tomassini-3}, getting that, on every compact complex manifold,
$$ \dim_\C \Tot^\bullet H^{\bullet,\bullet}_{BC}(X) + \dim_\C \Tot^\bullet H^{\bullet,\bullet}_{A}(X) \;\geq\; 2\, \dim_\C H^\bullet_{dR}(X;\C) \;, $$
and the characterization of the $\del\delbar$-Lemma in terms of the Bott-Chern cohomology given in \cite[Theorem B]{angella-tomassini-3}, namely, that the equality holds if and only if the $\del\delbar$-Lemma holds.

\medskip

Furthermore, Corollary \ref{cor:frolicher-like-double-complexes} and Theorem \ref{thm:caratt-deldelbar-lemma-double} allow also to study the cohomology of compact manifolds $X$ endowed with symplectic forms $\omega$. In this case, one considers the $\Z$-graded algebra $\wedge^\bullet X$ endowed with $\de \in \End^{1}\left(\wedge^\bullet X\right)$ and $\de^\Lambda := \left[\de,\, -\iota_{\omega^{-1}}\right] \;\in\; \End^{-1}\left(\wedge^\bullet X \right)$, which satisfy $\de^2=\left(\de^\Lambda\right)^2=\de\de^\Lambda+\de^\Lambda\de=0$. The symplectic Bott-Chern and Aeppli cohomologies have been introduced and studied by L.-S. Tseng and S.-T. Yau in \cite{tseng-yau-1, tseng-yau-2, tseng-yau-3}. In particular, we get the following result.

\smallskip
\noindent {\bfseries Theorem 3 (see Theorem \ref{thm:sympl}).\ }
{\itshape
Let $X$ be a compact manifold endowed with a symplectic structure $\omega$.
The inequality
\begin{equation}
\tag{\ref{eq:sympl}}
\dim_\R H^{\bullet}_{\left( \de , \de^\Lambda ; \de\de^\Lambda \right)}\left(X\right) + \dim_\R H^{\bullet}_{\left( \de\de^\Lambda ; \de , \de^\Lambda \right)}\left(X\right) \;\geq\; 2\, \dim_\R H^{\bullet}_{dR}(X;\R)
\end{equation}
holds. Furthermore, the equality in \eqref{eq:sympl} holds if and only if $X$ satisfies the Hard Lefschetz Condition.
}
\smallskip

We recall that a compact $2n$-dimensional manifold $X$ endowed with a symplectic form $\omega$ is said to satisfy the \emph{Hard Lefschetz Condition} if $\left[\omega\right]^k\smile \sspace \colon H^{n-k}_{dR}(X;\R) \to H^{n+k}_{dR}(X;\R)$ is an isomorphism for every $k\in\Z$.

\medskip

Finally, Corollary \ref{cor:frolicher-like-double-complexes} and Theorem \ref{thm:caratt-deldelbar-lemma-double} can be applied also to the study of the cohomology of generalized-complex manifolds. Generalized-complex geometry has been introduced by N. Hitchin in \cite{hitchin}, and studied, among others, by M. Gualtieri, \cite{gualtieri-phdthesis, gualtieri, gualtieri-kahler}, and G.~R. Cavalcanti, \cite{cavalcanti}. It provides a way to generalize both complex and symplectic geometry, since complex structures and symplectic structures appear as special cases of generalized-complex structures. See, e.g., \cite{hitchin-introduction} for an introduction to generalized-complex geometry; the cohomology of generalized-complex manifolds has been studied especially by G.~R. Cavalcanti, \cite{cavalcanti, cavalcanti-jgp, cavalcanti-computations}. On a manifold $X$ endowed with an $H$-twisted generalized complex structure $\mathcal{J}$, (see \S\ref{subsec:gen-cplx} for the definitions,) one can 
consider the $\Z$-graduation $\Tot \wedge^\bullet X\otimes_\R\C = \bigoplus_{k\in\Z} U^k_{\mathcal{J}}$, and the endomorphisms $\del_{\mathcal{J},H}\in\End^1\left(U^\bullet_{\mathcal{J}}\right)$ and $\delbar_{\mathcal{J},H}\in\End^{-1}\left(U^\bullet_{\mathcal{J}}\right)$, which satisfy $\del_{\mathcal{J},H}^2=\delbar_{\mathcal{J},H}^2=\del_{\mathcal{J},H}\delbar_{\mathcal{J},H}+\delbar_{\mathcal{J},H}\del_{\mathcal{J},H}=0$; then, let
$$ GH^{\bullet}_{\del_{\mathcal{J},H}}(X) \;:=\; \frac{\ker \del_{\mathcal{J},H}}{\imm \del_{\mathcal{J},H}} \;, \qquad GH^{\bullet}_{\delbar_{\mathcal{J},H}}(X) \;:=\; \frac{\ker \delbar_{\mathcal{J},H}}{\imm \delbar_{\mathcal{J},H}} \;, $$
and
$$ GH^{\bullet}_{BC_{\mathcal{J},H}}(X) \;:=\; \frac{\ker{\del_{\mathcal{J},H} \cap \ker \delbar_{\mathcal{J},H}}}{\imm \del_{\mathcal{J},H}\delbar_{\mathcal{J},H}} \;, \qquad GH^{\bullet}_{A_{\mathcal{J},H}}(X) \;:=\; \frac{\ker \del_{\mathcal{J},H}\delbar_{\mathcal{J},H}}{\imm \del_{\mathcal{J},H} + \imm \delbar_{\mathcal{J},H}} \;. $$
The above general results yield the following.

\smallskip
\noindent {\bfseries Theorem 4 (see Theorem \ref{thm:gen-frol-ineq} and Theorem \ref{thm:gen-charact}).\ }
{\itshape
 Let $X$ be a compact differentiable manifold endowed with an $H$-twisted generalized complex structure $\mathcal{J}$. Then
 \begin{equation}\tag{\ref{eq:ineq-frol-cplx-gen}}
  \dim_\C GH^{\bullet}_{BC_{\mathcal{J},H}}(X) + \dim_\C GH^{\bullet}_{A_{\mathcal{J},H}}(X) \;\geq\; \dim_\C GH^{\bullet}_{\delbar_{\mathcal{J},H}}(X) + \dim_\C GH^{\bullet}_{\del_{\mathcal{J},H}}(X) \;.
 \end{equation}
 Furthermore, $X$ satisfies the $\del_{\mathcal{J},H}\delbar_{\mathcal{J},H}$-Lemma if and only if the Hodge and Fr\"olicher spectral sequences associated to the canonical double complex $\left(U^{\bullet_1-\bullet_2}_{\mathcal{J}}\otimes\beta^{\bullet_2},\, \del_{\mathcal{J},H} \otimes_\C \id,\, \delbar_{\mathcal{J},H} \otimes_\C \beta\right)$ degenerate at the first level and the equality in \eqref{eq:ineq-frol-cplx-gen} holds.
}
\smallskip

% \medskip
% 
% \noindent{\sl Acknowledgments.} \dots

\section{Preliminaries and notation}

Fix $\K\in\{\R,\,\C\}$. In this section, we summarize some notation and results concerning graded $\K$-vector spaces endowed with two commuting differentials.

\subsection{(Bi-)graded vector spaces}

We set the notation, in constructing two functors in order to change over $\Z$-graduation and $\Z^2$-graduation of a $\K$-vector space.

\medskip

Consider a $\Z^2$-graded $\K$-vector space $A^{\bullet,\bullet}$ endowed with two endomorphisms $\delta_1 \in \End^{\hat\delta_{1,1}, \hat\delta_{1,2}}\left(A^{\bullet,\bullet}\right)$ and $\delta_2 \in \End^{\hat\delta_{2,1}, \hat\delta_{2,2}}\left(A^{\bullet,\bullet}\right)$ such that $\delta_1^2=\delta_2^2=\delta_1\delta_2+\delta_2\delta_1=0$. Define the $\Z$-graded $\K$-vector space
$$ \Tot^\bullet\left(A^{\bullet,\bullet}\right) \;:=\; \bigoplus_{p+q=\bullet} A^{p,q} \;, $$
endowed with the endomorphisms
$$ \delta_1 \in \End^{\hat\delta_{1,1}+\hat\delta_{1,2}}\left(\Tot^\bullet\left(A^{\bullet,\bullet}\right)\right) \qquad \text{ and } \qquad \delta_2 \in \End^{\hat\delta_{2,1}+\hat\delta_{2,2}}\left(\Tot^\bullet\left(A^{\bullet,\bullet}\right)\right) $$
such that $\delta_1^2=\delta_2^2=\delta_1\delta_2+\delta_2\delta_1=0$.

\medskip

Conversely, consider a $\Z$-graded $\K$-vector space $A^\bullet$ endowed with two endomorphisms $\delta_1 \in \End^{\hat\delta_1}\left(A^\bullet\right)$ and $\delta_2 \in \End^{\hat\delta_2}\left(A^\bullet\right)$ such that $\delta_1^2=\delta_2^2=\delta_1\delta_2+\delta_2\delta_1=0$. Following \cite[\S1.3]{brylinski}, \cite[\S4.2]{cavalcanti}, see \cite[\S II.2]{goodwillie}, \cite[\S II]{connes}, take an infinite cyclic multiplicative group $\left\{ \beta^m \st m \in \Z \right\}$ generated by some $\beta$, and consider the $\Z$-graded $\K$-vector space $\bigoplus_{\bullet\in\Z}\K\, \beta^\bullet$. Define the $\Z^2$-graded $\K$-vector space
$$ \Doub^{\bullet_1,\bullet_2}\left(A^{\bullet}\right) \;:=\; A^{\hat\delta_1\,\bullet_1 + \hat\delta_2\,\bullet_2} \otimes_\K \K\,\beta^{\bullet_2} \;, $$
endowed with the endomorphisms
$$ \delta_1\otimes_\K \id \in \End^{1,0}\left(\Doub^{\bullet,\bullet}\left(A^\bullet\right)\right) \qquad \text{ and } \qquad \delta_2\otimes_\K\beta \in \End^{1,0}\left(\Doub^{\bullet,\bullet}\left(A^\bullet\right)\right) \;, $$
which satisfy $\left(\delta_1\otimes_\K\id\right)^2=\left(\delta_2\otimes_\K\beta\right)^2=\left(\delta_1\otimes_\K\id\right)\left(\delta_2\otimes_\K\beta\right)+\left(\delta_2\otimes_\K\beta\right)\left(\delta_1\otimes_\K\id\right)=0$; following \cite[\S1.3]{brylinski}, \cite[\S4.2]{cavalcanti}, the double complex $\left(\Doub^{\bullet,\bullet}\left(A^{\bullet}\right),\, \delta_1\otimes_\K\id,\, \delta_2\otimes_\K\beta\right)$ is called the \emph{canonical double complex} associated to $A^{\bullet}$.

\subsection{Cohomologies}\label{subsec:cohom-complexes}

Let $A^\bullet$ be a $\Z$-graded $\mathbb{K}$-vector space endowed with two endomorphisms $\delta_1 \in \End^{\hat\delta_1}\left(A^\bullet\right)$ and $\delta_2 \in \End^{\hat\delta_2}\left(A^\bullet\right)$ such that
$$ \delta_1^2 \;=\; \delta_2^2 \;=\; \delta_1\delta_2+\delta_2\delta_1 \;=\; 0 \;. $$

Since one has the $\Z$-graded $\K$-vector sub-spaces $\imm\delta_1\delta_2 \subseteq \ker\delta_1 \cap \ker\delta_2$, and $\imm\delta_1 \subseteq \ker\delta_1$, and $\imm\delta_2 \subseteq \ker\delta_2$, and $\imm\delta_1+\imm\delta_2 \subseteq \ker\delta_1\delta_2$, one can define the $\Z$-graded $\K$-vector spaces
\begin{eqnarray*}
& H^\bullet_{\left( \delta_1 , \delta_2 ; \delta_1\delta_2 \right)}\left(A^\bullet\right) \;:=\; \frac{\ker\delta_1 \cap \ker\delta_2}{\imm\delta_1\delta_2} \;, & \\[5pt]
H^\bullet_{\left( \delta_1 ; \delta_1 \right)}\left(A^\bullet\right) \;:=\; \frac{\ker\delta_1}{\imm\delta_1} \;, & & H^\bullet_{\left( \delta_2 ; \delta_2 \right)}\left(A^\bullet\right) \;:=\; \frac{\ker\delta_2}{\imm\delta_2} \;, \\[5pt]
 & H^\bullet_{\left( \delta_1\delta_2 ; \delta_1 , \delta_2 \right)}\left(A^\bullet\right) \;:=\; \frac{\ker\delta_1\delta_2}{\imm\delta_1 + \imm\delta_2} \;, &
\end{eqnarray*}
and, since one has the $\K$-vector sub-space $\imm\left(\delta_1+\delta_2\right) \subseteq \ker\left(\delta_1+\delta_2\right)$, one can define the $\K$-vector space
$$ H_{\left( \delta_1+\delta_2 ; \delta_1+\delta_2 \right)}\left(\Tot A^\bullet\right) \;:=\; \frac{\ker\left(\delta_1+\delta_2\right)}{\imm\left(\delta_1+\delta_2\right)} \;; $$
we follow notation in \cite[Remark 5.16]{deligne-griffiths-morgan-sullivan}: more precisely, if maps $f_j \colon C_j\to A$ for $j\in\{1,\ldots,r\}$ and $g_k \colon A \to B_k$ for $k\in\{1,\ldots,s\}$ of $\K$-vector spaces are given, then $H_{ \left( f_1 , \ldots, f_r ; g_1 , \ldots, g_s \right) }$ denotes the quotient $\frac{\bigcap_{j=1}^{r} \ker f_j}{ \sum_{k=1}^{s} \imm g_k }$.
(Note that, up to consider $-\delta_2\in\End^{\hat\delta_2}\left(A^\bullet\right)$ instead of $\delta_2 \in \End^{\hat\delta_2}\left(A^\bullet\right)$, one has the $\K$-vector sub-space $\imm\left(\delta_1-\delta_2\right) \subseteq \ker\left(\delta_1-\delta_2\right)$, and hence one can consider also the $\K$-vector space
$H_{\left( \delta_1-\delta_2 ; \delta_1-\delta_2 \right)}\left(\Tot A^\bullet\right) := \frac{\ker\left(\delta_1-\delta_2\right)}{\imm\left(\delta_1-\delta_2\right)}$; note that, for $\sharp_{\delta_1,\delta_2}\in\left\{\left(\delta_1 ; \delta_1\right),\, \left(\delta_2 ; \delta_2\right),\, \left(\delta_1 , \delta_2 ; \delta_1\delta_2\right),\, \left(\delta_1\delta_2 ; \delta_1 , \delta_2\right)\right\}$, one has $H^\bullet_{\sharp_{\delta_1,\delta_2}}\left(A^\bullet\right)=H^\bullet_{\sharp_{\delta_1,-\delta_2}}\left(A^\bullet\right)$.)

\begin{rem}\label{rem:z-grad}
 Note that $H_{\left( \delta_1 + \delta_2 ; \delta_1 + \delta_2 \right)}\left(A^\bullet\right)$ admits a $\left(\left. \Z \middle\slash \left(\hat\delta_1-\hat\delta_2\right)\Z\right.\right)$-graduation;
 in particular, if $\hat\delta_1 = \hat\delta_2$, then $H^\bullet_{\left( \delta_1 + \delta_2 ; \delta_1 + \delta_2 \right)}\left(A^\bullet\right)$ is actually a $\Z$-graded $\K$-vector space.
\end{rem}

\begin{rem}\label{rem:z2-grad}
 Note that, for $\sharp \in \left\{ \left( \delta_1 , \delta_2 ; \delta_1\delta_2 \right) ,\, \left( \delta_1 ; \delta_1 \right) ,\, \left( \delta_2 ; \delta_2 \right) ,\, \left( \delta_1\delta_2 ; \delta_1 , \delta_2 \right) \right\}$, if $A^{\bullet,\bullet}$ is actually $\Z^2$-graded, then $H^\bullet_{\sharp}\left(A^\bullet\right)$ admits a $\Z^2$-graduation such that $\Tot^\bullet H^{\bullet,\bullet}_{\sharp}\left(A^{\bullet,\bullet}\right)=H^\bullet_{\sharp}\left(\Tot^\bullet A^{\bullet,\bullet}\right)$. Furthermore, for $\delta_1 \in \End^{\hat\delta_{1,1}, \hat\delta_{1,2}}\left(A^{\bullet,\bullet}\right)$ and $\delta_2 \in \End^{\hat\delta_{2,1}, \hat\delta_{2,2}}\left(A^{\bullet,\bullet}\right)$, one has that $H_{\left( \delta_1 + \delta_2 ; \delta_1 + \delta_2 \right)}\left(\Tot A^\bullet\right)$ admits a $\left(\left(\left. \Z \middle\slash \left(\hat\delta_{1,1}-\hat\delta_{2,1}\right)\Z\right.\right) \times \left(\left. \Z \middle\slash \left(\hat\delta_{1,2}-\hat\delta_{
2,2}\right)\Z\right.\right)\right)$-graduation; in particular, if $\hat\delta_{1,1}=\hat\delta_{2,1}$ and $\hat\delta_{1,2}=\hat\delta_{2,2}$, then $H_{\left( \delta_1 + \delta_2 ; \delta_1 + \delta_2 \right)}\left(\Tot A^\bullet\right)$ is actually $\Z^2$-graded.
\end{rem}

Since $\ker\delta_1\cap\ker\delta_2 \subseteq \ker\left(\delta_1\pm\delta_2\right)$ and $\imm\delta_1\delta_2\subseteq \imm\left(\delta_1\pm\delta_2\right)$ for $\pm\in\{+,-\}$, and $\ker\delta_1\cap\ker\delta_2 \subseteq \ker\delta_1$ and $\imm\delta_1\delta_2\subseteq \imm\delta_1$, and $\ker\delta_1\cap\ker\delta_2 \subseteq \ker\delta_2$ and $\imm\delta_1\delta_2\subseteq \imm\delta_2$, and $\ker\left(\delta_1\pm\delta_2\right)\subseteq\ker\delta_1\delta_2$ and $\imm\left(\delta_1\pm\delta_2\right)\subseteq\imm\delta_1+\imm\delta_2$ for $\pm\in\{+,-\}$, and $\ker\delta_1\subseteq\ker\delta_1\delta_2$ and $\imm\delta_1\subseteq\imm\delta_1+\imm\delta_2$, and $\ker\delta_2\subseteq\ker\delta_1\delta_2$ and $\imm\delta_2\subseteq\imm\delta_1+\imm\delta_2$, then the identity map induces natural morphisms of (possibly $\Z$-graded, possibly $\Z^2$-graded) $\K$-vector spaces
\begin{footnotesize}
$$
\xymatrix{
&& H^\bullet_{\left( \delta_1 , \delta_2 ; \delta_1\delta_2 \right)}\left(A^\bullet\right) \ar@/_1.5pc/[lld] \ar@/_1pc/[ld] \ar@/^1pc/[rd] \ar@/^1.5pc/[rrd] \ar[dd] && \\
H^\bullet_{\left( \delta_1 ; \delta_1 \right)}\left(A^\bullet\right) \ar@/_1.5pc/[rrd] & H_{\left( \delta_1+\delta_2 ; \delta_1+\delta_2 \right)}\left(\Tot A^\bullet\right) \ar@/_1pc/[rd] & & H_{\left( \delta_1-\delta_2 ; \delta_1-\delta_2 \right)}\left(\Tot A^\bullet\right) \ar@/^1pc/[ld] & H^\bullet_{\left( \delta_2 ; \delta_2 \right)}\left(A^\bullet\right) \ar@/^1.5pc/[lld] \\
&& H^\bullet_{\left( \delta_1\delta_2 ; \delta_1 , \delta_2 \right)}\left(A^\bullet\right) &&
}
$$
\end{footnotesize}

(As a matter of notation, by writing, for example, $H^\bullet_{\left( \delta_1 , \delta_2 ; \delta_1\delta_2 \right)}\left(A^\bullet\right) \to H_{\left( \delta_1+\delta_2 ; \delta_1+\delta_2 \right)}\left(\Tot A^\bullet\right)$, we mean $\Tot H^\bullet_{\left( \delta_1 , \delta_2 ; \delta_1\delta_2 \right)}\left(A^\bullet\right) \to H_{\left( \delta_1+\delta_2 ; \delta_1+\delta_2 \right)}\left(\Tot A^\bullet\right)$.)

\subsection{$\delta_1\delta_2$-Lemma}

Let $A^\bullet$ be a $\Z$-graded $\mathbb{K}$-vector space endowed with two endomorphisms $\delta_1 \in \End^{\hat\delta_1}\left(A^\bullet\right)$ and $\delta_2 \in \End^{\hat\delta_2}\left(A^\bullet\right)$ such that $\delta_1^2 = \delta_2^2 = \delta_1\delta_2+\delta_2\delta_1 = 0$, and consider the cohomologies introduced in \S\ref{subsec:cohom-complexes}.
In general, the natural maps induced by the identity between such cohomologies are neither injective nor surjective: the following definition, \cite{deligne-griffiths-morgan-sullivan}, points out when they are actually isomorphisms.

\begin{defi}[{\cite{deligne-griffiths-morgan-sullivan}}]
 A $\Z$-graded $\mathbb{K}$-vector space $A^\bullet$ endowed with two endomorphisms $\delta_1 \in \End^{\hat\delta_1}\left(A^\bullet\right)$ and $\delta_2 \in \End^{\hat\delta_2}\left(A^\bullet\right)$ such that $\delta_1^2 = \delta_2^2 = \delta_1\delta_2+\delta_2\delta_1 = 0$ is said to satisfy the \emph{$\delta_1\delta_2$-Lemma} if and only if
 $$ \ker\delta_1 \cap \ker \delta_2 \cap \left(\imm\delta_1 + \imm\delta_2\right) \;=\; \imm\delta_1\delta_2 \;,$$
 namely, if and only if the natural map $H^\bullet_{\left( \delta_1 , \delta_2 ; \delta_1\delta_2 \right)}\left(A^\bullet\right) \to H^\bullet_{\left( \delta_1\delta_2 ; \delta_1 , \delta_2 \right)}\left(A^\bullet\right)$ induced by the identity is injective.

 A $\Z^2$-graded $\mathbb{K}$-vector space $A^{\bullet,\bullet}$ endowed with two endomorphisms $\delta_1 \in \End^{\hat\delta_{1,1}, \hat\delta_{1,2}}\left(A^{\bullet,\bullet}\right)$ and $\delta_2 \in \End^{\hat\delta_{2,1}, \hat\delta_{2,2}}\left(A^{\bullet,\bullet}\right)$ such that $\delta_1^2 = \delta_2^2 = \delta_1\delta_2+\delta_2\delta_1 = 0$ is said to satisfy the \emph{$\delta_1\delta_2$-Lemma} if and only if $\Tot^\bullet\left(A^{\bullet,\bullet}\right)$ satisfies the $\delta_1\delta_2$-Lemma.
\end{defi}

We recall the following result, which provides further characterizations of the validity of the $\delta_1\delta_2$-Lemma.
(Note that, according to Remark \ref{rem:z-grad} and Remark \ref{rem:z2-grad}, the natural maps induced by the identity in Lemma \ref{lemma:equiv} are maps of possibly $\Z$-graded, possibly $\Z^2$-graded $\K$-vector spaces.)

\begin{lemma}[{see \cite[Lemma 5.15]{deligne-griffiths-morgan-sullivan}}]\label{lemma:equiv}
 Let $A^\bullet$ be a $\Z$-graded $\mathbb{K}$-vector space endowed with two endomorphisms $\delta_1 \in \End^{\hat\delta_1}\left(A^\bullet\right)$ and $\delta_2 \in \End^{\hat\delta_2}\left(A^\bullet\right)$ such that $\delta_1^2 = \delta_2^2 = \delta_1\delta_2+\delta_2\delta_1 = 0$. The following conditions are equivalent:
 \begin{enumerate}
  \item\label{item:equiv-1} $A^\bullet$ satisfies the $\delta_1\delta_2$-Lemma, namely, the natural map $H^\bullet_{\left( \delta_1 , \delta_2 ; \delta_1\delta_2 \right)}\left(A^\bullet\right) \to H^\bullet_{\left( \delta_1\delta_2 ; \delta_1 , \delta_2 \right)}\left(A^\bullet\right)$ induced by the identity is injective;
  \item\label{item:equiv-2} the natural map $H^\bullet_{\left( \delta_1 , \delta_2 ; \delta_1\delta_2 \right)}\left(A^\bullet\right) \to H^\bullet_{\left( \delta_1\delta_2 ; \delta_1 , \delta_2 \right)}\left(A^\bullet\right)$ induced by the identity is surjective;
  \item\label{item:equiv-3} both the natural map $H^\bullet_{\left( \delta_1 , \delta_2 ; \delta_1\delta_2 \right)}\left(A^\bullet\right) \to H^\bullet_{\left( \delta_1 ; \delta_1 \right)}\left(A^\bullet\right)$ induced by the identity and the natural map $H^\bullet_{\left( \delta_1 , \delta_2 ; \delta_1\delta_2 \right)}\left(A^\bullet\right) \to H^\bullet_{\left( \delta_2 ; \delta_2 \right)}\left(A^\bullet\right)$ induced by the identity are injective;
  \item\label{item:equiv-4} both the natural map $H^\bullet_{\left( \delta_1 ; \delta_1 \right)}\left(A^\bullet\right) \to H^\bullet_{\left( \delta_1\delta_2 ; \delta_1 , \delta_2 \right)}\left(A^\bullet\right)$ induced by the identity and the natural map $H^\bullet_{\left( \delta_2 ; \delta_2 \right)}\left(A^\bullet\right) \to H^\bullet_{\left( \delta_1\delta_2 ; \delta_1 , \delta_2 \right)}\left(A^\bullet\right)$ induced by the identity are surjective.
 \end{enumerate}

 Furthermore, suppose that the $\K$-vector space $\ker\delta_1\delta_2$ admits a $\Z$-graduation
 $$ \ker\delta_1\delta_2 \;=\; \bigoplus_{\ell\in\Z} \left(\ker\delta_1\delta_2 \cap \tilde A^\ell\right) $$
 with respect to which $\ker\left(\delta_1 \pm \delta_2\right) \cap \tilde A^\bullet = \left(\ker\delta_1\cap\ker\delta_2\right) \cap \tilde A^\bullet$. (For example, if $\hat\delta_1\neq\hat\delta_2$, then take the $\Z$-graduation given by $A^\bullet$. For example, if $A^{\bullet,\bullet}$ is actually $\Z^2$-graded and $\delta_1\in\End^{\hat\delta_{1,1},\hat\delta_{1,2}}\left(A^{\bullet,\bullet}\right)$ and $\delta_2\in\End^{\hat\delta_{2,1},\hat\delta_{2,2}}\left(A^{\bullet,\bullet}\right)$ with $\left(\hat\delta_{1,1},\hat\delta_{1,2}\right)\neq \left(\hat\delta_{2,1},\hat\delta_{2,2}\right)$, then take the $\Z$-graduation induced by the $\Z^2$-graduation of $A^{\bullet,\bullet}$ by means of a chosen bijection $\Z\stackrel{\simeq}{\to}\Z^2$.) Then the previous conditions are equivalent to each of the following:
 \begin{enumerate}\setcounter{enumi}{4}
  \item\label{item:equiv-5} the natural map $\Tot H^\bullet_{\left( \delta_1 , \delta_2 ; \delta_1\delta_2 \right)}\left(A^\bullet\right) \to H_{\left( \delta_1+\delta_2 ; \delta_1+\delta_2 \right)}\left(\Tot A^\bullet\right)$ induced by the identity is injective;
  \item\label{item:equiv-6} the natural map $H_{\left( \delta_1+\delta_2 ; \delta_1+\delta_2 \right)}\left(\Tot A^\bullet\right) \to \Tot H^\bullet_{\left( \delta_1\delta_2 ; \delta_1 , \delta_2 \right)}\left(A^\bullet\right)$ induced by the identity is surjective;
\item\label{item:equiv-7} the natural map $\Tot H^\bullet_{\left( \delta_1 , \delta_2 ; \delta_1\delta_2 \right)}\left(A^\bullet\right) \to H_{\left( \delta_1-\delta_2 ; \delta_1-\delta_2 \right)}\left(\Tot A^\bullet\right)$ induced by the identity is injective;
  \item\label{item:equiv-8} the natural map $H_{\left( \delta_1-\delta_2 ; \delta_1-\delta_2 \right)}\left(\Tot A^\bullet\right) \to \Tot H^\bullet_{\left( \delta_1\delta_2 ; \delta_1 , \delta_2 \right)}\left(A^\bullet\right)$ induced by the identity is surjective.
 \end{enumerate}
\end{lemma}

\begin{proof}
For the sake of completeness, we recall here the proof in \cite{deligne-griffiths-morgan-sullivan}.

\paragrafoo{\eqref{item:equiv-1} $\Rightarrow$ \eqref{item:equiv-3}} By the hypothesis, $\ker\delta_1 \cap \ker\delta_2 \cap \left(\imm\delta_1 + \imm\delta_2\right) = \imm\delta_1\delta_2$, and we have to prove that $\ker\delta_2 \cap \imm\delta_1 \subseteq \imm\delta_1\delta_2$ and $\ker\delta_1 \cap \imm\delta_2 \subseteq \imm\delta_1\delta_2$. Since $\imm\delta_1 \subseteq \imm\delta_1 + \imm\delta_2$ and $\imm\delta_2 \subseteq \imm\delta_1 + \imm\delta_2$, one gets immediately that the natural maps $H^\bullet_{\left( \delta_1 , \delta_2 ; \delta_1\delta_2 \right)}\left(A^\bullet\right) \to H^\bullet_{\left( \delta_1 ; \delta_1 \right)}\left(A^\bullet\right)$ and $H^\bullet_{\left( \delta_1 , \delta_2 ; \delta_1\delta_2 \right)}\left(A^\bullet\right) \to H^\bullet_{\left( \delta_2 ; \delta_2 \right)}\left(A^\bullet\right)$ are injective.

\paragrafoo{\eqref{item:equiv-3} $\Rightarrow$ \eqref{item:equiv-4}} By the hypotheses, we have that $\ker\delta_2 \cap \imm\delta_1 = \imm\delta_1\delta_2$ and $\ker\delta_1 \cap \imm\delta_2 = \imm\delta_1\delta_2$, and we have to prove that $\ker\delta_1 + \imm\delta_2 \supseteq \ker\delta_1\delta_2$ and $\ker\delta_2 + \imm\delta_1 \supseteq \ker\delta_1\delta_2$. Let $x\in\ker\delta_1\delta_2$. Then $\delta_1(x) \in \ker\delta_2 \cap \imm\delta_1 = \imm\delta_1\delta_2$: let $y \in A^{\bullet}$ be such that $\delta_1(x) = \delta_1\delta_2(y)$. Then $x = \left(x - \delta_2(y)\right) + \delta_2(y) \in \ker\delta_1 + \imm\delta_2$, since $\delta_1\left(x - \delta_2(y)\right) = 0$; it follows that the natural map $H^\bullet_{\left( \delta_1 ; \delta_1 \right)}\left(A^\bullet\right) \to H^\bullet_{\left( \delta_1\delta_2 ; \delta_1 , \delta_2 \right)}\left(A^\bullet\right)$ is surjective. Analogously, $\delta_2(x) \in \ker\delta_1 \cap \imm\delta_2 = \imm\delta_1\delta_2$: let $z$ be such that $\delta_2(x) = 
\delta_1\delta_
2(z)$. Then $x = \left(x + \delta_1(z)\right) - \delta_1(z) \in \ker\delta_2 + \imm\delta_1$, since $\delta_2\left(x + \delta_1(z)\right) = 0$; it follows that the natural map $H^\bullet_{\left( \delta_2 ; \delta_2 \right)}\left(A^\bullet\right) \to H^\bullet_{\left( \delta_1\delta_2 ; \delta_1 , \delta_2 \right)}\left(A^\bullet\right)$ is surjective.

\paragrafoo{\eqref{item:equiv-4} $\Rightarrow$ \eqref{item:equiv-2}} By the hypothesis, $\ker\delta_1 + \imm\delta_2 = \ker\delta_1\delta_2$ and $\ker\delta_2 + \imm\delta_1 = \ker\delta_1\delta_2$, and we have to prove that $\left(\ker\delta_1 \cap \ker\delta_2\right) + \imm\delta_1 + \imm\delta_2 \supseteq \ker\delta_1\delta_2$. Since $\ker\delta_1\delta_2 = \left(\ker\delta_1 + \imm\delta_2\right) \cap \left(\ker\delta_2 + \imm\delta_1\right) \subseteq \left(\ker\delta_1 \cap \ker\delta_2\right) + \imm\delta_1 + \imm\delta_2$, one gets that the natural map $H^\bullet_{\left( \delta_1 , \delta_2 ; \delta_1\delta_2 \right)}\left(A^\bullet\right) \to H^\bullet_{\left( \delta_1\delta_2 ; \delta_1 , \delta_2 \right)}\left(A^\bullet\right)$ is surjective.

\paragrafoo{\eqref{item:equiv-2} $\Rightarrow$ \eqref{item:equiv-1}} By the hypothesis, $\left(\ker\delta_1 \cap \ker\delta_2\right) + \imm\delta_1 + \imm\delta_2 = \ker\delta_1\delta_2$, and we have to prove that $\ker \delta_1 \cap \ker\delta_2 \cap \left(\imm\delta_1 + \imm\delta_2\right) \subseteq \imm\delta_1\delta_2$. Let $x :=: \delta_1(y) + \delta_2(z) \in \ker\delta_1 \cap \ker\delta_2 \cap \left(\imm\delta_1 + \imm\delta_2\right)$. Therefore $y \in \ker\delta_1\delta_2 = \left(\ker\delta_1 \cap \ker\delta_2\right) + \imm\delta_1 + \imm\delta_2$ and $z \in \ker\delta_1\delta_2 = \left(\ker\delta_1 \cap \ker\delta_2\right) + \imm\delta_1 + \imm\delta_2$. It follows that $\delta_1(y) \in \imm\delta_1\delta_2$ and $\delta_2(z) \in \imm\delta_1\delta_2$, and hence $x = \delta_1(y) + \delta_2(z) \in \imm\delta_1\delta_2$, proving that the natural map $H^\bullet_{\left( \delta_1 , \delta_2 ; \delta_1\delta_2 \right)}\left(A^\bullet\right) \to H^\bullet_{\left( \delta_1\delta_
2 ; \delta_1 , \delta_2 \right)}\left(A^\bullet\right)$ is injective.

\paragrafoo{\eqref{item:equiv-1} $\Rightarrow$ \eqref{item:equiv-5}, and \eqref{item:equiv-1} $\Rightarrow$ \eqref{item:equiv-7}} By the hypothesis, $\ker\delta_1 \cap \ker\delta_2 \cap \left(\imm\delta_1 + \imm\delta_2\right) = \imm\delta_1\delta_2$, and we have to prove that $\ker\delta_1 \cap \ker\delta_2 \cap \imm\left(\delta_1 \pm \delta_2\right) \subseteq \imm\delta_1\delta_2$ for $\pm\in\{+,-\}$. Since $\ker\delta_1 \cap \ker\delta_2 \cap \imm\left(\delta_1 \pm \delta_2\right) \subseteq \ker\delta_1 \cap \ker\delta_2 \cap \left(\imm\delta_1 + \imm\delta_2\right)$, one gets immediately that the natural map $\Tot H^\bullet_{\left( \delta_1 , \delta_2 ; \delta_1\delta_2 \right)}\left(A^\bullet\right) \to H_{\left( \delta_1 \pm \delta_2 ; \delta_1 \pm \delta_2 \right)}\left(A^\bullet\right)$ is injective.

\paragrafoo{\eqref{item:equiv-5} $\Rightarrow$ \eqref{item:equiv-6}, and \eqref{item:equiv-7} $\Rightarrow$ \eqref{item:equiv-8}} Fix $\pm\in\{+,-\}$. By the hypothesis, $\ker\delta_1 \cap \ker\delta_2 \cap \imm \left(\delta_1 \pm \delta_2\right) = \imm\delta_1\delta_2$, and we have to prove that $\ker\left( \delta_1 \pm \delta_2 \right) + \imm\delta_1 + \imm\delta_2 \supseteq \ker\delta_1\delta_2$. 
Let $x\in \ker\delta_1\delta_2$. Then $\left(\delta_1 \pm \delta_2\right)(x) \in \ker\delta_1 \cap \ker\delta_2 \cap \imm\left(\delta_1 \pm \delta_2\right) = \imm\delta_1\delta_2$; let $z\in\Tot A^{\bullet}$ be such that $\left(\delta_1 \pm \delta_2\right)(x) = \delta_1\delta_2(z)$. Since $\left(\delta_1 \pm \delta_2\right) \left(x \pm \frac{1}{2}\, \delta_1(z) - \frac{1}{2}\, \delta_2(z)\right) = 0$, one gets that $x = \left(x \pm \frac{1}{2}\, \delta_1(z) - \frac{1}{2}\, \delta_2(z)\right) - \left(\pm \frac{1}{2}\, \delta_1(z)\right) + \frac{1}{2}\, \delta_2(z) \in \ker\left( \delta_1 \pm \delta_2 \right) + \imm\delta_1 + \imm\delta_2$, proving that the natural map $H_{\left( \delta_1 \pm \delta_2 ; \delta_1 \pm \delta_2 \right)}\left(\Tot A^\bullet\right) \to \Tot H^\bullet_{\left( \delta_1\delta_2 ; \delta_1 , \delta_2 \right)}\left(A^\bullet\right)$ is surjective.

To conclude the equivalences, we assume the additional hypothesis given in the statement.

\paragrafoo{\eqref{item:equiv-6} $\Rightarrow$ \eqref{item:equiv-2}, and \eqref{item:equiv-8} $\Rightarrow$ \eqref{item:equiv-2}} Fix $\pm\in\{+,-\}$. By the hypothesis, $\ker\left( \delta_1 \pm \delta_2 \right) + \imm\delta_1 + \imm\delta_2 = \ker\delta_1\delta_2$, and we have to prove that $\left(\ker\delta_1 \cap \ker\delta_2\right) + \imm\delta_1 + \imm\delta_2 \supseteq \ker\delta_1\delta_2$. By the additional hypothesis, we have that $\ker\delta_1\delta_2$ admits a $\Z$-graduation $\ker\delta_1\delta_2 = \bigoplus_{\ell\in\Z} \left(\ker\delta_1\delta_2 \cap \tilde A^\ell\right)$ with respect to which $\ker\left(\delta_1 \pm \delta_2\right) \cap \tilde A^\bullet = \left(\ker\delta_1\cap\ker\delta_2\right) \cap \tilde A^\bullet$. Then one has that
\begin{eqnarray*}
\ker\delta_1\delta_2 &=& \bigoplus_{\ell\in\Z} \left(\ker\delta_1\delta_2 \cap \tilde A^\ell\right)
\;=\; \bigoplus_{\ell\in\Z} \left( \left(\ker\left( \delta_1 \pm \delta_2 \right) + \imm\delta_1 + \imm\delta_2\right) \cap \tilde A^\ell\right) \\[5pt]
&\subseteq& \bigoplus_{\ell\in\Z} \left( \left(\ker\left( \delta_1 \pm \delta_2 \right) \cap \tilde A^\ell\right) + \imm\delta_1 + \imm\delta_2\right)
\;=\; \bigoplus_{\ell\in\Z} \left( \left(\left(\ker \delta_1 \cap \delta_2 \right) \cap \tilde A^\ell\right) + \imm\delta_1 + \imm\delta_2\right) \\[5pt]
&\subseteq& \left(\ker \delta_1 \cap \delta_2 \right) + \imm\delta_1 + \imm\delta_2 \;,
\end{eqnarray*}
proving that the natural map $H_{\left( \delta_1 \pm \delta_2 ; \delta_1 \pm \delta_2 \right)}\left(\Tot A^\bullet\right) \to \Tot H^\bullet_{\left( \delta_1\delta_2 ; \delta_1 , \delta_2 \right)}\left(A^\bullet\right)$ is surjective.
\end{proof}

\medskip

By noting that, for $\sharp_{\delta_1,\delta_2}\in\left\{\delta_1,\, \delta_2,\, \delta_1\delta_2,\, \delta_1+\delta_2,\, \delta_1-\delta_2\right\}$,
$$ \left(\ker \sharp_{\delta_1\otimes_\K\id, \delta_2\otimes_\K\beta}\right)^{\bullet_1,\bullet_2} \;=\; \left(\ker \sharp_{\delta_1,\delta_2}\right)^{\hat\delta_1\, \bullet_1 + \hat\delta_2\, \bullet_2} \otimes_\K \K\, \beta^{\bullet_2} $$
and
$$ \left(\imm \sharp_{\delta_1\otimes_\K\id, \delta_2\otimes_\K\beta}\right)^{\bullet_1,\bullet_2} \;=\; \left(\imm \sharp_{\delta_1,\delta_2}\right)^{\hat\delta_1\, \bullet_1 + \hat\delta_2\, \bullet_2} \otimes_\K \K\, \beta^{\bullet_2} \;, $$
we get the following lemmata.

\begin{lemma}\label{lemma:cohom-a-doub}
 Let $A^\bullet$ be a $\Z$-graded $\mathbb{K}$-vector space endowed with two endomorphisms $\delta_1 \in \End^{\hat\delta_1}\left(A^\bullet\right)$ and $\delta_2 \in \End^{\hat\delta_2}\left(A^\bullet\right)$ such that $\delta_1^2 = \delta_2^2 = \delta_1\delta_2+\delta_2\delta_1 = 0$.
 Then, there are natural isomorphisms of $\K$-vector spaces
 $$ H^{\bullet_1,\bullet_2}_{\sharp_{\delta_1\otimes_\K\id, \delta_2\otimes_\K\beta}} \left(\Doub^{\bullet,\bullet}A^{\bullet}\right) \;\simeq\; \Doub^{\bullet_1,\bullet_2} H_{\sharp_{\delta_1,\delta_2}}^{\bullet}\left(A^\bullet\right) \;, $$
 where $\sharp_{\delta_1,\delta_2} \in \left\{ (\delta_1 , \delta_2 ; \delta_1\delta_2) ,\, (\delta_1 ; \delta_1) ,\, (\delta_2 ; \delta_2) ,\, (\delta_1\delta_2 ; \delta_1 , \delta_2) \right\}$.
\end{lemma}

\begin{lemma}\label{lemma:deldelbar-lemma-a-doub}
 Let $A^\bullet$ be a $\Z$-graded $\mathbb{K}$-vector space endowed with two endomorphisms $\delta_1 \in \End^{\hat\delta_1}\left(A^\bullet\right)$ and $\delta_2 \in \End^{\hat\delta_2}\left(A^\bullet\right)$ such that $\delta_1^2 = \delta_2^2 = \delta_1\delta_2+\delta_2\delta_1 = 0$. Denote the greatest common divisor of $\hat\delta_1$ and $\hat\delta_2$ by $\GCD{\hat\delta_1}{\hat\delta_2}$.
 The following conditions are equivalent:
 \begin{enumerate}
  \item\label{item:deldelbar-lemma-ab-1} $A^{\GCD{\hat\delta_1}{\hat\delta_2} \, \bullet}$ satisfies the $\delta_1\delta_2$-Lemma;
  \item\label{item:deldelbar-lemma-ab-2} $\Doub^{\bullet,\bullet}\left(A^{\bullet}\right)$ satisfies the $\left(\delta_1\otimes_\K\id\right)\left(\delta_2\otimes_\K\beta\right)$-Lemma.
 \end{enumerate}
\end{lemma}

\begin{proof}
Indeed,
$$ \left( \ker\left(\delta_1\otimes_\K\id\right) \cap \imm\left(\delta_2\otimes_\K\beta\right) \right)^{\bullet_1,\bullet_2} \;=\; \left(\ker\delta_1 \cap \imm\delta_2 \cap A^{\hat\delta_1\, \bullet_1 + \hat\delta_2\, \bullet_2}\right) \otimes_\K \K\, \beta^{\bullet_2} $$
and
$$ \left(\imm\left(\delta_1\otimes_\K\id\right)\left(\delta_2\otimes_\K\beta\right)\right)^{\bullet_1,\bullet_2} \;=\;  \left(\imm\delta_1\delta_2 \cap A^{\hat\delta_1\, \bullet_1 + \hat\delta_2\, \bullet_2}\right) \otimes_\K \K\, \beta^{\bullet_2} \;,$$
completing the proof.
\end{proof}

\section{An inequality {\itshape à la} Fr\"olicher}

Let $A^{\bullet,\bullet}$ be a bounded $\Z^2$-graded $\K$-vector space endowed with two endomorphisms $\delta_1 \in \End^{1,0}\left(A^{\bullet,\bullet}\right)$ and $\delta_2 \in \End^{0,1}\left(A^{\bullet,\bullet}\right)$ such that $\delta_1^2 = \delta_2^2 = \delta_1\delta_2+\delta_2\delta_1 = 0$.
The bi-grading induces two natural bounded filtrations of the $\Z$-graded $\K$-vector space $\Tot^\bullet \left( A^{\bullet,\bullet} \right)$ endowed with the endomorphism $\delta_1+\delta_2\in \End^{1}\left(\Tot^\bullet \left( A^{\bullet,\bullet} \right)\right)$, namely,
$$ \left\{ {'F}^{p} \Tot^\bullet \left( A^{\bullet,\bullet} \right) := \bigoplus_{\substack{r+s=\bullet\\r\geq p}}A^{r,s} \hookrightarrow \Tot^\bullet\left(A^{\bullet,\bullet}\right) \right\}_{p\in\Z} $$
and
$$ \left\{ {''F}^{q} \Tot^\bullet \left( A^{\bullet,\bullet} \right) := \bigoplus_{\substack{r+s=\bullet\\s\geq q}}A^{r,s} \hookrightarrow \Tot^\bullet\left(A^{\bullet,\bullet}\right) \right\}_{q\in\Z} \;. $$
Such filtrations induce naturally two spectral sequences, respectively, 
$$ \left\{{'E}_r^{\bullet,\bullet}\left(A^{\bullet,\bullet},\, \delta_1,\, \delta_2\right)\right\}_{r\in\Z} \qquad \text{ and } \qquad \left\{{''E}_r^{\bullet,\bullet}\left(A^{\bullet,\bullet},\, \delta_1,\, \delta_2\right)\right\}_{r\in\Z} \;, $$
such that
$$ {'E}_1^{\bullet_1,\bullet_2}\left(A^{\bullet,\bullet},\, \delta_1,\, \delta_2\right) \;\simeq\; H^{\bullet_1,\bullet_2}_{\left( \delta_2 ; \delta_2 \right)}\left(A^{\bullet,\bullet}\right) \;\Rightarrow\; H^{\bullet_1+\bullet_2}_{\left( \delta_1+\delta_2 ; \delta_1+\delta_2 \right)}\left(\Tot^\bullet \left( A^{\bullet,\bullet} \right)\right) \;, $$
and 
$$ {''E}_1^{\bullet_1,\bullet_2}\left(A^{\bullet,\bullet},\, \delta_1,\, \delta_2\right) \;\simeq\; H^{\bullet_1,\bullet_2}_{\left( \delta_1 ; \delta_1 \right)}\left(A^{\bullet,\bullet}\right) \;\Rightarrow\; H^{\bullet_1+\bullet_2}_{\left( \delta_1+\delta_2 ; \delta_1+\delta_2 \right)}\left(\Tot^\bullet\left(A^{\bullet,\bullet}\right)\right) \;, $$
see, e.g., \cite[\S2.4]{mccleary}, see also \cite[\S3.5]{griffiths-harris}.

\medskip

By using these spectral sequences (and up to consider $-\delta_2$ instead of $\delta_2$), one gets the classical \emph{A. Fr\"olicher inequality}.

\begin{notation}
 Given two $\Z$-graded $\K$-vector spaces $A^\bullet$ and $B^\bullet$, writing, for example, $\dim_\K A^\bullet \geq \dim_\K B^\bullet$, we mean that, for any $k\in\Z$, the inequality $\dim_\K A^k \geq \dim_\K B^k$ holds.
\end{notation}

\begin{prop}[{\cite[Theorem 2]{frolicher}}]\label{prop:frolicher-classico}
 Let $A^{\bullet,\bullet}$ be a bounded $\Z^2$-graded $\K$-vector space endowed with two endomorphisms $\delta_1 \in \End^{1,0}\left(A^{\bullet,\bullet}\right)$ and $\delta_2 \in \End^{0,1}\left(A^{\bullet,\bullet}\right)$ such that $\delta_1^2 = \delta_2^2 = \delta_1\delta_2+\delta_2\delta_1 = 0$. Then, for $\pm\in\{+,-\}$,
 $$ \min\left\{ \dim_\K \Tot^\bullet H^{\bullet,\bullet}_{\left( \delta_1 ; \delta_1 \right)}\left(A^{\bullet,\bullet}\right) ,\; \dim_\K \Tot^\bullet H^{\bullet,\bullet}_{\left( \delta_2 ; \delta_2 \right)}\left(A^{\bullet,\bullet}\right) \right\} \;\geq\; \dim_\K H^\bullet_{\left( \delta_1 \pm \delta_2 ; \delta_1 \pm \delta_2 \right)}\left(\Tot^\bullet A^{\bullet,\bullet}\right) \;. $$
\end{prop}

As a straightforward consequence, the following result holds in the $\Z$-graded case.
\begin{cor}
 Let $A^{\bullet}$ be a bounded $\Z$-graded $\K$-vector space endowed with two endomorphisms $\delta_1 \in \End^{\hat\delta_1}\left(A^{\bullet}\right)$ and $\delta_2 \in \End^{\hat\delta_2}\left(A^{\bullet}\right)$ such that $\delta_1^2 = \delta_2^2 = \delta_1\delta_2+\delta_2\delta_1 = 0$. Then, for $\pm\in\{+,-\}$,
 \begin{eqnarray*}
  \lefteqn{\min\left\{ \sum_{p+q=\bullet}\dim_\K H^{\hat\delta_1 \, p + \hat\delta_2 \, q}_{\left( \delta_1 ; \delta_1 \right)}\left(A^{\bullet}\right) ,\; \sum_{p+q=\bullet} \dim_\K H^{\hat\delta_1 \, p + \hat\delta_2 \, q}_{\left( \delta_2 ; \delta_2 \right)}\left(A^{\bullet}\right) \right\}} \\[5pt]
  &\geq& \dim_\K H^\bullet_{\left( \left(\delta_1\otimes_\K\id\right) \pm \left(\delta_2\otimes_\K\beta\right) ; \left(\delta_1\otimes_\K\id\right) \pm \left(\delta_2\otimes_\K\beta\right) \right)}\left(\Tot^\bullet \Doub^{\bullet,\bullet}A^\bullet\right) \;.
 \end{eqnarray*}
\end{cor}

\begin{proof}
 By Lemma \ref{lemma:cohom-a-doub}, one has that, for $\sharp_{\delta_1,\delta_2} \in \left\{ \left( \delta_1 , \delta_2 ; \delta_1\delta_2 \right) ,\, \left( \delta_1 ; \delta_1 \right) ,\, \left( \delta_2 ; \delta_2 \right) ,\, \left( \delta_1\delta_2 ; \delta_1 , \delta_2 \right) \right\}$,
 $$ \dim_\K H^{\bullet_1,\bullet_2}_{\sharp_{\delta_1\otimes_\K\id, \delta_2\otimes_\K\beta}} \left(\Doub^{\bullet,\bullet}A^{\bullet}\right) \;=\; \dim_\K H_{\sharp_{\delta_1,\delta_2}}^{\hat\delta_1 \, \bullet_1 + \hat\delta_2 \, \bullet_2}\left(A^\bullet\right) \;. $$
 Hence, by applying the classical Fr\"olicher inequality, Proposition \ref{prop:frolicher-classico}, to $\Doub^{\bullet,\bullet}$ endowed with $\delta_1\otimes_\K\id$ and $\delta_2\otimes_\K\beta$, one gets, for $\pm\in\{+,-\}$,
 \begin{eqnarray*}
  \lefteqn{\min\left\{ \sum_{p+q=\bullet}\dim_\K H^{\hat\delta_1 \, p + \hat\delta_2 \, q}_{\left( \delta_1 ; \delta_1 \right)}\left(A^{\bullet}\right) ,\; \sum_{p+q=\bullet} \dim_\K H^{\hat\delta_1 \, p + \hat\delta_2 \, q}_{\left( \delta_2 ; \delta_2 \right)}\left(A^{\bullet}\right) \right\}} \\[5pt]
  &=& \min\left\{ \sum_{p+q=\bullet} \dim_\K H^{p,q}_{\left( \delta_1\otimes_\K\id ; \delta_1\otimes_\K\id \right)}\left(\Doub^{\bullet,\bullet} A^{\bullet}\right) ,\; \sum_{p+q=\bullet} \dim_\K H^{p,q}_{\left( \delta_2\otimes_\K\beta ; \delta_2\otimes_\K\beta \right)}\left(\Doub^{\bullet,\bullet} A^{\bullet}\right) \right\} \\[5pt]
  &=& \min\left\{ \dim_\K \Tot^\bullet H^{\bullet,\bullet}_{\left( \delta_1\otimes_\K\id ; \delta_1\otimes_\K\id \right)}\left(\Doub^{\bullet,\bullet} A^{\bullet}\right) ,\; \dim_\K \Tot^\bullet H^{\bullet,\bullet}_{\left( \delta_2\otimes_\K\beta ; \delta_2\otimes_\K\beta \right)}\left(\Doub^{\bullet,\bullet} A^{\bullet}\right) \right\} \\[5pt]
  &\geq& \dim_\K H^\bullet_{\left( \left(\delta_1\otimes_\K\id\right) \pm \left(\delta_2\otimes_\K\beta\right) ; \left(\delta_1\otimes_\K\id\right) \pm \left(\delta_2\otimes_\K\beta\right) \right)}\left(\Tot^\bullet \Doub^{\bullet,\bullet} A^\bullet\right) \;,
\end{eqnarray*}
 completing the proof.
\end{proof}

We prove the following inequality {\itshape à la} Fr\"olicher involving the cohomologies $H^\bullet_{\left( \delta_1 , \delta_2 ; \delta_1\delta_2 \right)}\left(A^\bullet\right)$ and $H^\bullet_{\left( \delta_1\delta_2 ; \delta_1 , \delta_2 \right)}\left(A^\bullet\right)$, other than $H^\bullet_{\left( \delta_1 ; \delta_1 \right)}\left(A^\bullet\right)$ and $H^\bullet_{\left( \delta_2 ; \delta_2 \right)}\left(A^\bullet\right)$.

\begin{thm}\label{thm:disug-frol}
Let $A^\bullet$ be a $\Z$-graded $\mathbb{K}$-vector space endowed with two endomorphisms $\delta_1 \in \End^{\hat\delta_1}\left(A^\bullet\right)$ and $\delta_2 \in \End^{\hat\delta_2}\left(A^\bullet\right)$ such that $\delta_1^2 = \delta_2^2 = \delta_1\delta_2+\delta_2\delta_1 = 0$.
Suppose that
$$ \dim_\K H^\bullet_{\left( \delta_1 ; \delta_1 \right)}\left(A^\bullet\right) < +\infty  \qquad \text{ and } \qquad  \dim_\K H^\bullet_{\left( \delta_2 ; \delta_2 \right)}\left(A^\bullet\right) \;<\; +\infty \;.$$
Then
\begin{equation}\label{eq:disug-frol}
\dim_\K H^\bullet_{\left( \delta_1 , \delta_2 ; \delta_1\delta_2 \right)}\left(A^\bullet\right) + \dim_\K H^\bullet_{\left( \delta_1\delta_2 ; \delta_1 , \delta_2 \right)}\left(A^\bullet\right) \;\geq\; \dim_\K H^\bullet_{\left( \delta_1 ; \delta_1 \right)}\left(A^\bullet\right) + \dim_\K H^\bullet_{\left( \delta_2 ; \delta_2 \right)}\left(A^\bullet\right) \;. 
\end{equation}
\end{thm}

\begin{proof}
If either $H^\bullet_{\left( \delta_1 , \delta_2 ; \delta_1\delta_2 	\right)}\left(A^\bullet\right)$ or $H^\bullet_{\left( \delta_1\delta_2 ; \delta_1 , \delta_2 \right)}\left(A^\bullet\right)$ is not finite-dimensional, then the inequality holds trivially; hence, we are reduced to suppose that also
$$ \dim_\K H^\bullet_{\left( \delta_1 , \delta_2 ; \delta_1\delta_2 	\right)}\left(A^\bullet\right) < +\infty \qquad \text{ and } \qquad \dim_\K H^\bullet_{\left( \delta_1\delta_2 ; \delta_1 , \delta_2 \right)}\left(A^\bullet\right) < +\infty \;.$$

Following J. Varouchas, \cite[\S3.1]{varouchas}, consider the exact sequences
 \begin{eqnarray*}
  0 \to \frac{\imm\delta_1 \cap \imm\delta_2}{\imm\delta_1\delta_2} \to \frac{\ker\delta_1 \cap \imm\delta_2}{\imm\delta_1\delta_2} \to \frac{\ker\delta_1}{\imm\delta_1} \to \frac{\ker\delta_1\delta_2}{\imm\delta_1 + \imm\delta_2} \to \frac{\ker\delta_1\delta_2}{\ker\delta_1 + \imm\delta_2} \to 0 \;, \\[5pt]
  0 \to \frac{\imm\delta_1 \cap \imm\delta_2}{\imm\delta_1\delta_2} \to \frac{\ker\delta_2 \cap \imm\delta_1}{\imm\delta_1\delta_2} \to \frac{\ker\delta_2}{\imm\delta_2} \to \frac{\ker\delta_1\delta_2}{\imm\delta_1 + \imm\delta_2} \to \frac{\ker\delta_1\delta_2}{\ker\delta_2 + \imm\delta_1} \to 0 \;, \\[5pt]
  0 \to \frac{\imm\delta_1 \cap \ker\delta_2}{\imm\delta_1\delta_2} \to \frac{\ker\delta_1 \cap \ker\delta_2}{\imm\delta_1\delta_2} \to \frac{\ker\delta_1}{\imm\delta_1} \to \frac{\ker\delta_1\delta_2}{\ker\delta_2 + \imm\delta_1} \to \frac{\ker\delta_1\delta_2}{\ker\delta_1 + \ker\delta_2} \to 0 \;, \\[5pt]
  0 \to \frac{\imm\delta_2 \cap \ker\delta_1}{\imm\delta_1\delta_2} \to \frac{\ker\delta_1 \cap \ker\delta_2}{\imm\delta_1\delta_2} \to \frac{\ker\delta_2}{\imm\delta_2} \to \frac{\ker\delta_1\delta_2}{\ker\delta_1 + \imm\delta_2} \to \frac{\ker\delta_1\delta_2}{\ker\delta_1 + \ker\delta_2} \to 0
 \end{eqnarray*}
of $\Z$-graded $\K$-vector spaces.

Note that all the $\K$-vector spaces appearing in the exact sequences have finite dimension. Indeed, since $H^\bullet_{\left( \delta_1\delta_2 ; \delta_1 , \delta_2 \right)}\left(A^\bullet\right)$ has finite dimension, then
$$ \dim_\K \frac{\ker\delta_1\delta_2}{\ker\delta_1 + \imm\delta_2} \;<\; +\infty \qquad \text{ and } \qquad \dim_\K \frac{\ker\delta_1\delta_2}{\ker\delta_2 + \imm\delta_1} \;<\; +\infty \;;$$
since $H^\bullet_{\left( \delta_1 , \delta_2 ; \delta_1 \delta_2 \right)}\left(A^\bullet\right)$ has finite dimension, then
$$ \dim_\K \frac{\imm\delta_1 \cap \ker\delta_2}{\imm\delta_1\delta_2} \;<\; +\infty \qquad \text{ and } \qquad \dim_\K \frac{\imm\delta_2 \cap \ker\delta_1}{\imm\delta_1\delta_2} \;<\; +\infty \;.$$
Furthermore, note that the natural maps $\frac{\ker\delta_1 \cap \imm\delta_2}{\imm\delta_1\delta_2} \to H^{\bullet}_{\left( \delta_1 , \delta_2 ; \delta_1 \delta_2 \right)}\left(A^\bullet\right)$ and $\frac{\ker\delta_2 \cap \imm\delta_1}{\imm\delta_1\delta_2} \to H^{\bullet}_{\left( \delta_1 , \delta_2 ; \delta_1 \delta_2 \right)}\left(A^\bullet\right)$ induced by the identity are injective, and hence
$$ \dim_\K \frac{\ker\delta_1 \cap \imm\delta_2}{\imm\delta_1\delta_2} \;<\; +\infty \qquad \text{ and } \qquad \dim_\K \frac{\ker\delta_2 \cap \imm\delta_1}{\imm\delta_1\delta_2} \;<\; +\infty \;;$$
it follows also that
$$ \dim_\K \frac{\imm\delta_1 \cap \imm\delta_2}{\imm\delta_1\delta_2} \;<\; +\infty \;. $$
Analogously, since the natural maps $H^\bullet_{\left( \delta_1\delta_2 ; \delta_1 , \delta_2 \right)}\left(A^\bullet\right) \to \frac{\ker\delta_1\delta_2}{\ker\delta_2 + \imm\delta_1}$ and $H^\bullet_{\left( \delta_1\delta_2 ; \delta_1 , \delta_2 \right)}\left(A^\bullet\right) \to \frac{\ker\delta_1\delta_2}{\ker\delta_1 + \imm\delta_2}$ induced by the identity are surjective, then
$$ \dim_\K \frac{\ker\delta_1\delta_2}{\ker\delta_2 + \imm\delta_1} \;<\; +\infty \qquad \text{ and } \qquad \dim_\K \frac{\ker\delta_1\delta_2}{\ker\delta_1 + \imm\delta_2} \;<\; +\infty \;,$$
and hence also
$$ \dim_\K \frac{\ker\delta_1\delta_2}{\ker\delta_1 + \ker\delta_2} \;<\; +\infty \;. $$

By using the above exact sequences, it follows that
 \begin{eqnarray*}
  \dim_\K \frac{\ker\delta_1\delta_2}{\imm\delta_1 + \imm\delta_2} &=& \dim_\K \frac{\imm\delta_1 \cap \imm\delta_2}{\imm\delta_1\delta_2} - \dim_\K \frac{\ker\delta_1 \cap \imm\delta_2}{\imm\delta_1\delta_2} + \dim_\K \frac{\ker\delta_1}{\imm\delta_1} + \dim_\K \frac{\ker\delta_1\delta_2}{\ker\delta_1 + \imm\delta_2} \;, \\[5pt]
  \dim_\K \frac{\ker\delta_1\delta_2}{\imm\delta_1 + \imm\delta_2}  &=& \dim_\K \frac{\imm\delta_1 \cap \imm\delta_2}{\imm\delta_1\delta_2} - \dim_\K \frac{\ker\delta_2 \cap \imm\delta_1}{\imm\delta_1\delta_2} + \dim_\K \frac{\ker\delta_2}{\imm\delta_2} + \dim_\K \frac{\ker\delta_1\delta_2}{\ker\delta_2 + \imm\delta_1} \;, \\[5pt]
  \dim_\K \frac{\ker\delta_1 \cap \ker\delta_2}{\imm\delta_1\delta_2} &=& \dim_\K \frac{\imm\delta_1 \cap \ker\delta_2}{\imm\delta_1\delta_2} + \dim_\K \frac{\ker\delta_1}{\imm\delta_1} - \dim_\K \frac{\ker\delta_1\delta_2}{\ker\delta_2 + \imm\delta_1} + \dim_\K \frac{\ker\delta_1\delta_2}{\ker\delta_1 + \ker\delta_2} \;, \\[5pt]
  \dim_\K \frac{\ker\delta_1 \cap \ker\delta_2}{\imm\delta_1\delta_2} &=& \dim_\K \frac{\imm\delta_2 \cap \ker\delta_1}{\imm\delta_1\delta_2} + \dim_\K \frac{\ker\delta_2}{\imm\delta_2} - \dim_\K \frac{\ker\delta_1\delta_2}{\ker\delta_1 + \imm\delta_2} + \dim_\K \frac{\ker\delta_1\delta_2}{\ker\delta_1 + \ker\delta_2} \;,
 \end{eqnarray*}
from which, by summing up, one gets
 \begin{eqnarray}
  \lefteqn{ 2\, \dim_\K \frac{\ker\delta_1\delta_2}{\imm\delta_1 + \imm\delta_2} + 2\, \dim_\K \frac{\ker\delta_1 \cap \ker\delta_2}{\imm\delta_1\delta_2} } \nonumber\\[5pt]
  &=& 2\, \dim_\K \frac{\ker\delta_1}{\imm\delta_1} + 2\, \dim_\K \frac{\ker\delta_2}{\imm\delta_2} + 2\, \dim_\K \frac{\imm\delta_1 \cap \imm\delta_2}{\imm\delta_1\delta_2}  + 2\, \dim_\K \frac{\ker\delta_1\delta_2}{\ker\delta_1 + \ker\delta_2} \label{eq:BC-A-delta1-delta2}\\[5pt]
  &\geq& 2\, \dim_\K \frac{\ker\delta_1}{\imm\delta_1} + 2\, \dim_\K \frac{\ker\delta_2}{\imm\delta_2} \;, \nonumber
 \end{eqnarray}
 yielding
 $$
  \dim_\K \frac{\ker\delta_1\delta_2}{\imm\delta_1 + \imm\delta_2} + \dim_\K \frac{\ker\delta_1 \cap \ker\delta_2}{\imm\delta_1\delta_2} \; \geq \; \dim_\K \frac{\ker\delta_1}{\imm\delta_1} + \dim_\K \frac{\ker\delta_2}{\imm\delta_2} \;,
 $$
 and hence the theorem.
\end{proof}

\begin{rem}
 Note that the proof of Theorem \ref{thm:disug-frol} works also for $\Z^2$-graded $\K$-vector spaces, since in this case J. Varouchas' exact sequences are in fact exact sequences of $\Z^2$-graded $\K$-vector spaces. More precisely, one gets that, given a $\Z^2$-graded $\mathbb{K}$-vector space $A^{\bullet,\bullet}$ endowed with two endomorphisms $\delta_1 \in \End^{\hat\delta_{1,1}, \hat\delta_{1,2}}\left(A^{\bullet,\bullet}\right)$ and $\delta_2 \in \End^{\hat\delta_{2,1}, \hat\delta_{2,2}}\left(A^{\bullet,\bullet}\right)$ such that $\delta_1^2 = \delta_2^2 = \delta_1\delta_2+\delta_2\delta_1 = 0$, and supposed that
 $$ \dim_\K H^{\bullet,\bullet}_{\left( \delta_1 ; \delta_1 \right)}\left(A^{\bullet,\bullet}\right) < +\infty  \qquad \text{ and } \qquad  \dim_\K H^{\bullet,\bullet}_{\left( \delta_2 ; \delta_2 \right)}\left(A^{\bullet,\bullet}\right) \;<\; +\infty \;,$$
 then
 \begin{equation*}
 \dim_\K H^{\bullet,\bullet}_{\left( \delta_1 , \delta_2 ; \delta_1\delta_2 \right)}\left(A^{\bullet,\bullet}\right) + \dim_\K H^{\bullet,\bullet}_{\left( \delta_1\delta_2 ; \delta_1 , \delta_2 \right)}\left(A^{\bullet,\bullet}\right) \;\geq\; \dim_\K H^{\bullet,\bullet}_{\left( \delta_1 ; \delta_1 \right)}\left(A^{\bullet,\bullet}\right) + \dim_\K H^{\bullet,\bullet}_{\left( \delta_2 ; \delta_2 \right)}\left(A^{\bullet,\bullet}\right) \;. 
 \end{equation*}
\end{rem}

As a consequence of Theorem \ref{thm:disug-frol} and Proposition \ref{prop:frolicher-classico}, one gets the following inequality {\itshape à la} Fr\"olicher for double complexes, namely, $\Z^2$-graded $\K$-vector spaces $B^{\bullet,\bullet}$ endowed with two endomorphisms $\delta_1 \in \End^{1,0}\left(B^{\bullet,\bullet}\right)$ and $\delta_2 \in \End^{0,1}\left(B^{\bullet,\bullet}\right)$ such that $\delta_1^2 = \delta_2^2 = \delta_1\delta_2+\delta_2\delta_1 = 0$.

\begin{cor}\label{cor:frolicher-like-double-complexes}
 Let $B^{\bullet,\bullet}$ be a bounded $\Z^2$-graded $\mathbb{K}$-vector space endowed with two endomorphisms $\delta_1 \in \End^{1,0}\left(B^{\bullet,\bullet}\right)$ and $\delta_2 \in \End^{0,1}\left(B^{\bullet,\bullet}\right)$ such that $\delta_1^2 = \delta_2^2 = \delta_1\delta_2+\delta_2\delta_1 = 0$.
 Suppose that
 $$ \dim_\K \Tot^\bullet H^{\bullet,\bullet}_{\left( \delta_1 ; \delta_1 \right)}\left(B^{\bullet,\bullet}\right) < +\infty  \qquad \text{ and } \qquad  \dim_\K \Tot^\bullet H^{\bullet,\bullet}_{\left( \delta_2 ; \delta_2 \right)}\left(B^{\bullet,\bullet}\right) \;<\; +\infty \;.$$
 Then, for $\pm\in\{+,-\}$,
 \begin{equation}\label{eq:disug-frol-double-complexes}
 \dim_\K \Tot^\bullet H^{\bullet,\bullet}_{\left( \delta_1 , \delta_2 ; \delta_1\delta_2 \right)}\left(B^{\bullet,\bullet}\right) + \dim_\K \Tot^\bullet H^{\bullet,\bullet}_{\left( \delta_1\delta_2 ; \delta_1 , \delta_2 \right)}\left(B^{\bullet,\bullet}\right) \;\geq\; 2\, \dim_\K H^\bullet_{\left( \delta_1 \pm \delta_2 ; \delta_1 \pm \delta_2 \right)}\left(\Tot^\bullet B^{\bullet,\bullet}\right) \;. 
 \end{equation}
\end{cor}

\section{A characterization of $\delta_1\delta_2$-Lemma by means of the inequality {\itshape à la} Fr\"olicher}

With the aim to characterize the validity of the $\delta_1\delta_2$-Lemma in terms of the dimensions of the cohomologies $H^\bullet_{\left( \delta_1 , \delta_2 ; \delta_1\delta_2 \right)}\left(A^\bullet\right)$ and $H^\bullet_{\left( \delta_1\delta_2 ; \delta_1 , \delta_2 \right)}\left(A^\bullet\right)$, we need the following lemmata.

\begin{lemma}\label{lemma:BC-dr-surj}
 Let $B^{\bullet,\bullet}$ be a $\Z^2$-graded $\mathbb{K}$-vector space endowed with two endomorphisms $\delta_1 \in \End^{1,0}\left(B^{\bullet,\bullet}\right)$ and $\delta_2 \in \End^{0,1}\left(B^{\bullet,\bullet}\right)$ such that $\delta_1^2 = \delta_2^2 = \delta_1\delta_2+\delta_2\delta_1 = 0$. If
 $$ \frac{\imm\delta_1 \cap \imm\delta_2}{\imm\delta_1\delta_2} \;=\; \{0\} \;, $$
 then the natural map $\iota\colon \Tot^\bullet H^{\bullet,\bullet}_{\left( \delta_1 , \delta_2 ; \delta_1\delta_2 \right)}\left(B^{\bullet,\bullet}\right) \to H^\bullet_{\left( \delta_1+\delta_2 ; \delta_1+\delta_2 \right)}\left(\Tot^\bullet B^{\bullet,\bullet}\right)$ induced by the identity is surjective.
\end{lemma}

\begin{proof}
Let $\mathfrak{a} :=: \left[x\right] \in H^\bullet_{\left( \delta_1+\delta_2 ; \delta_1+\delta_2 \right)}\left(\Tot^\bullet B^{\bullet,\bullet}\right)$. Since $\delta_1(x)+\delta_2(x)=0$ and $\imm\delta_1 \cap \imm\delta_2 = \imm\delta_1\delta_2$, then we have  $\delta_1(x) = -\delta_2(x) \in \imm\delta_1\cap\imm\delta_2 \;=\; \imm\delta_1\delta_2$;
let $y \in \Tot^{\bullet-1} B^{\bullet,\bullet}$ be such that $\delta_1(x) = \delta_1\delta_2(y) \;=\; -\delta_2(x)$.
Hence, consider $\mathfrak{a} = \left[x\right] = \left[x-\left(\delta_1+\delta_2\right)(y)\right] \in H^\bullet_{\left( \delta_1+\delta_2 ; \delta_1+\delta_2 \right)}\left(\Tot^\bullet B^{\bullet,\bullet}\right)$,
and note that
$\mathfrak{a} = \iota\left(\left[x-\left(\delta_1+\delta_2\right)(y)\right]\right)$ where $\left[x-\left(\delta_1+\delta_2\right)(y)\right] \in \Tot^\bullet H^{\bullet,\bullet}_{\left( \delta_1 , \delta_2 ; \delta_1\delta_2 \right)}\left(B^{\bullet,\bullet}\right)$,
since $\delta_1\left(x-\left(\delta_1+\delta_2\right)(y)\right)=0$ and $\delta_2\left(x-\left(\delta_1+\delta_2\right)(y)\right)=0$.
\end{proof}

\begin{lemma}\label{lemma:dR-A-inj}
 Let $B^{\bullet,\bullet}$ be a $\Z^2$-graded $\mathbb{K}$-vector space endowed with two endomorphisms $\delta_1 \in \End^{1,0}\left(B^{\bullet,\bullet}\right)$ and $\delta_2 \in \End^{0,1}\left(B^{\bullet,\bullet}\right)$ such that $\delta_1^2 = \delta_2^2 = \delta_1\delta_2+\delta_2\delta_1 = 0$. If
 $$ \frac{\ker\delta_1\delta_2}{\ker\delta_1 + \ker\delta_2} \;=\; \{0\} \;, $$
 then the natural map $\iota\colon H^\bullet_{\left( \delta_1+\delta_2 ; \delta_1+\delta_2 \right)}\left(\Tot^\bullet B^{\bullet,\bullet}\right) \to \Tot^\bullet H^{\bullet,\bullet}_{\left( \delta_1\delta_2 ; \delta_1 , \delta_2 \right)}\left(B^{\bullet,\bullet}\right)$ induced by the identity is injective.
\end{lemma}

\begin{proof}
Let $\mathfrak{a} :=: \left[x\right] \in H^\bullet_{\left( \delta_1+\delta_2 ; \delta_1+\delta_2 \right)}\left(\Tot^\bullet B^{\bullet,\bullet}\right)$.
Suppose that $\iota(\mathfrak{a}) = \left[0\right] \in \Tot^\bullet H^{\bullet,\bullet}_{\left( \delta_1\delta_2 ; \delta_1 , \delta_2 \right)}\left(B^{\bullet,\bullet}\right)$, that is, there exist $y \in \Tot^{\bullet-1}B^{\bullet,\bullet}$ and $z \in \Tot^{\bullet-1}B^{\bullet,\bullet}$ such that $x = \delta_1(y) + \delta_2(z)$. Since $\left(\delta_1+\delta_2\right)(x)=0$ and $\ker\delta_1\delta_2 = \ker\delta_1 + \ker\delta_2$, it follows that $\delta_1\delta_2\left(z-y\right) = 0$, that is, $z-y \in \ker\delta_1\delta_2 = \ker\delta_1 + \ker\delta_2$. Let $u \in \ker\delta_1$ and $v \in \ker\delta_2$ be such that $z-y = u+v$. Then, one has that $x = \delta_1(y) + \delta_2(z) = \delta_1(y) + \delta_2(y+u+v) = \left(\delta_1+\delta_2\right)(y+u) \in \imm\left(\delta_1+\delta_2\right)$, proving that $\mathfrak{a} = \left[0\right] \in H^\bullet_{\left( \delta_1+\delta_2 ; \delta_1+\delta_2 \right)}\left(\Tot^\bullet B^{\bullet,\bullet}\right)$.
\end{proof}

We can now prove the following characterization of the $\delta_1\delta_2$-Lemma for double complexes in terms of the equality in \eqref{eq:disug-frol-double-complexes}.

\begin{thm}\label{thm:caratt-deldelbar-lemma-double}
 Let $B^{\bullet,\bullet}$ be a bounded $\Z^2$-graded $\mathbb{K}$-vector space endowed with two endomorphisms $\delta_1 \in \End^{1,0}\left(B^{\bullet,\bullet}\right)$ and $\delta_2 \in \End^{0,1}\left(B^{\bullet,\bullet}\right)$ such that $\delta_1^2 = \delta_2^2 = \delta_1\delta_2+\delta_2\delta_1 = 0$.
 Suppose that
 $$ \dim_\K H^{\bullet,\bullet}_{\left( \delta_1 ; \delta_1 \right)}\left(B^{\bullet,\bullet}\right) < +\infty  \qquad \text{ and } \qquad  \dim_\K H^{\bullet,\bullet}_{\left( \delta_2 ; \delta_2 \right)}\left(B^{\bullet,\bullet}\right) \;<\; +\infty \;.$$
 The following conditions are equivalent:
 \begin{enumerate}
  \item\label{item:caratt-bi-1} $B^{\bullet,\bullet}$ satisfies the $\delta_1\delta_2$-Lemma;
  \item\label{item:caratt-bi-2} the equality
        \begin{eqnarray*}
        \lefteqn{ \dim_\K \Tot^{\bullet} H^{\bullet,\bullet}_{\left( \delta_1 , \delta_2 ; \delta_1\delta_2 \right)}\left(B^{\bullet,\bullet}\right) + \dim_\K \Tot^{\bullet} H^{\bullet,\bullet}_{\left( \delta_1\delta_2 ; \delta_1 , \delta_2 \right)}\left(B^{\bullet,\bullet}\right) } \\[5pt]
        &=& 2\, \dim_\K H^{\bullet}_{\left( \delta_1+\delta_2 ; \delta_1+\delta_2 \right)}\left(\Tot^\bullet B^{\bullet,\bullet}\right) \;.
        \end{eqnarray*}
        holds.
 \end{enumerate}
\end{thm}

\begin{proof}
\paragrafoo{\eqref{item:caratt-bi-1} $\Rightarrow$ \eqref{item:caratt-bi-2}}
By Lemma \ref{lemma:equiv}, it follows that 
$$ \dim_\K \Tot^{\bullet} H^{\bullet,\bullet}_{\left( \delta_1 , \delta_2 ; \delta_1\delta_2 \right)}\left(B^{\bullet,\bullet}\right) \;\leq\; \dim_\K H^{\bullet}_{\left( \delta_1+\delta_2 ; \delta_1+\delta_2 \right)}\left(\Tot^\bullet B^{\bullet,\bullet}\right) $$
and 
$$ \dim_\K \Tot^{\bullet} H^{\bullet,\bullet}_{\left( \delta_1\delta_2 ; \delta_1 , \delta_2 \right)}\left(B^{\bullet,\bullet}\right) \;\leq\; \dim_\K H^{\bullet}_{\left( \delta_1+\delta_2 ; \delta_1+\delta_2 \right)}\left(\Tot^\bullet B^{\bullet,\bullet}\right) \;.$$
By Corollary \ref{cor:frolicher-like-double-complexes}, it follows that
$$ \dim_\K \Tot^\bullet H^{\bullet,\bullet}_{\left( \delta_1 , \delta_2 ; \delta_1\delta_2 \right)}\left(B^{\bullet,\bullet}\right) + \dim_\K \Tot^\bullet H^{\bullet,\bullet}_{\left( \delta_1\delta_2 ; \delta_1 , \delta_2 \right)}\left(B^{\bullet,\bullet}\right) \;\geq\; 2\, \dim_\K H^\bullet_{\left( \delta_1 + \delta_2 ; \delta_1 + \delta_2 \right)}\left(\Tot^\bullet B^{\bullet,\bullet}\right) \;.$$
Hence actually the equality holds.

\paragrafoo{\eqref{item:caratt-bi-2} $\Rightarrow$ \eqref{item:caratt-bi-1}}
Since, by \eqref{eq:BC-A-delta1-delta2} and Proposition \ref{prop:frolicher-classico}, it holds
\begin{eqnarray*}
\lefteqn{\dim_\K \Tot^{\bullet} H^{\bullet,\bullet}_{\left( \delta_1 , \delta_2 ; \delta_1\delta_2 \right)}\left(B^{\bullet,\bullet}\right) + \dim_\K \Tot^{\bullet} H^{\bullet,\bullet}_{\left( \delta_1\delta_2 ; \delta_1 , \delta_2 \right)}\left(B^{\bullet,\bullet}\right)} \\[5pt]
 &=& \dim_\K \Tot^{\bullet} H^{\bullet,\bullet}_{\left( \delta_1 ; \delta_1 \right)}\left(B^{\bullet,\bullet}\right) + \dim_\K \Tot^{\bullet} H^{\bullet,\bullet}_{\left( \delta_2 ; \delta_2 \right)}\left(B^{\bullet,\bullet}\right) \\[5pt]
 && + \dim_\K \frac{\imm\delta_1 \cap \imm\delta_2}{\imm\delta_1\delta_2} + \dim_\K \frac{\ker\delta_1\delta_2}{\ker\delta_1 + \ker\delta_2} \\[5pt]
 &\geq& 2\, \dim_\K H^{\bullet}_{\left( \delta_1+\delta_2 ; \delta_1+\delta_2 \right)}\left(\Tot^\bullet B^{\bullet,\bullet}\right) \;,
\end{eqnarray*}
then, by the hypothesis, it follows that
$$ \frac{\imm\delta_1 \cap \imm\delta_2}{\imm\delta_1\delta_2} \;=\; \left\{0\right\} \qquad \text{ and } \qquad \frac{\ker\delta_1\delta_2}{\ker\delta_1 + \ker\delta_2} \;=\; \left\{0\right\} \;.$$
By Lemma \ref{lemma:BC-dr-surj}, one gets that the natural map $H^\bullet_{\left( \delta_1 , \delta_2 ; \delta_1\delta_2 \right)}\left(\Tot^\bullet B^{\bullet,\bullet}\right) \to H^\bullet_{\left( \delta_1+\delta_2 ; \delta_1+\delta_2 \right)}\left(\Tot^\bullet B^{\bullet,\bullet}\right)$ induced by the identity is surjective; by Lemma \ref{lemma:dR-A-inj}, one gets that the natural map $H^\bullet_{\left( \delta_1+\delta_2 ; \delta_1+\delta_2 \right)}\left(\Tot^\bullet B^{\bullet,\bullet}\right) \to H^\bullet_{\left( \delta_1\delta_2 ; \delta_1 , \delta_2 \right)}\left(\Tot^\bullet B^{\bullet,\bullet}\right)$ induced by the identity is injective. In particular, one has that
$$ \dim_\K \Tot^{\bullet} H^{\bullet,\bullet}_{\left( \delta_1 , \delta_2 ; \delta_1\delta_2 \right)}\left(B^{\bullet,\bullet}\right) \;\geq\; \dim_\K H^{\bullet}_{\left( \delta_1+\delta_2 ; \delta_1+\delta_2 \right)}\left(\Tot^\bullet B^{\bullet,\bullet}\right) $$
and that
$$ \dim_\K \Tot^{\bullet} H^{\bullet,\bullet}_{\left( \delta_1\delta_2 ; \delta_1 , \delta_2 \right)}\left(B^{\bullet,\bullet}\right) \;\geq\; \dim_\K H^{\bullet}_{\left( \delta_1+\delta_2 ; \delta_1+\delta_2 \right)}\left(\Tot^\bullet B^{\bullet,\bullet}\right) \;.$$
Hence, by the hypothesis, it holds in fact that
$$ \dim_\K \Tot^{\bullet} H^{\bullet,\bullet}_{\left( \delta_1 , \delta_2 ; \delta_1\delta_2 \right)}\left(B^{\bullet,\bullet}\right) \;=\; \dim_\K H^{\bullet}_{\left( \delta_1+\delta_2 ; \delta_1+\delta_2 \right)}\left(\Tot^\bullet B^{\bullet,\bullet}\right) \;=\; \dim_\K \Tot^{\bullet} H^{\bullet,\bullet}_{\left( \delta_1\delta_2 ; \delta_1 , \delta_2 \right)}\left(B^{\bullet,\bullet}\right) \;.$$
Since $H^{\bullet,\bullet}_{\left( \delta_1 , \delta_2 ; \delta_1\delta_2 \right)}\left(B^{\bullet,\bullet}\right)$ and $H^{\bullet,\bullet}_{\left( \delta_1\delta_2 ; \delta_1 , \delta_2 \right)}\left(B^{\bullet,\bullet}\right)$ are finite-dimensional by hypothesis, it follows that the natural maps $H^\bullet_{\left( \delta_1 , \delta_2 ; \delta_1\delta_2 \right)}\left(\Tot^\bullet B^{\bullet,\bullet}\right) \to H^\bullet_{\left( \delta_1+\delta_2 ; \delta_1+\delta_2 \right)}\left(\Tot^\bullet B^{\bullet,\bullet}\right)$ and $H^\bullet_{\left( \delta_1+\delta_2 ; \delta_1+\delta_2 \right)}\left(\Tot^\bullet B^{\bullet,\bullet}\right) \to H^\bullet_{\left( \delta_1\delta_2 ; \delta_1 , \delta_2 \right)}\left(\Tot^\bullet B^{\bullet,\bullet}\right)$ induced by the identity are in fact isomorphisms. By Lemma \ref{lemma:equiv}, one gets the theorem.
\end{proof}

\medskip

In order to apply Theorem \ref{thm:caratt-deldelbar-lemma-double} to $\Z$-graded $\K$-vector spaces to get geometric applications, e.g., for compact symplectic manifolds, we need to record the following corollaries.

\begin{cor}\label{cor:charact-hodge-frolicher-double-1}
 Let $A^\bullet$ be a bounded $\Z$-graded $\mathbb{K}$-vector space endowed with two endomorphisms $\delta_1 \in \End^{\hat\delta_1}\left(A^\bullet\right)$ and $\delta_2 \in \End^{\hat\delta_2}\left(A^\bullet\right)$ such that $\delta_1^2 = \delta_2^2 = \delta_1\delta_2+\delta_2\delta_1 = 0$. Denote the greatest common divisor of $\hat\delta_1$ and $\hat\delta_2$ by $\GCD{\hat\delta_1}{\hat\delta_2}$.
 Suppose that
 $$ \dim_\K H^{\bullet}_{\left( \delta_1 ; \delta_1 \right)}\left(A^{\bullet}\right) < +\infty  \qquad \text{ and } \qquad  \dim_\K H^{\bullet}_{\left( \delta_2 ; \delta_2 \right)}\left(A^{\bullet}\right) \;<\; +\infty \;.$$
 The following conditions are equivalent:
 \begin{enumerate}
  \item $A^{\GCD{\hat\delta_1}{\hat\delta_2}\,\bullet}$ satisfies the $\delta_1\delta_2$-Lemma;
  \item the equality
        \begin{eqnarray*}
        \lefteqn{\sum_{p+q=\bullet} \left( \dim_\K H^{\hat\delta_1 \, p + \hat\delta_2 \, q}_{\left( \delta_1 , \delta_2 ; \delta_1\delta_2 \right)}\left(A^{\bullet}\right) + \dim_\K H^{\hat\delta_1 \, p + \hat\delta_2 \, q}_{\left( \delta_1\delta_2 ; \delta_1 , \delta_2 \right)}\left(A^{\bullet}\right) \right)} \\[5pt]
        &=& 2\, \dim_\K H^\bullet_{\left( \delta_1\otimes_\K\id + \delta_2\otimes_\K\beta ; \delta_1\otimes_\K\id + \delta_2\otimes_\K\beta \right)}\left(\Tot^\bullet\Doub^{\bullet,\bullet}A^\bullet\right) \;.
        \end{eqnarray*}
        holds.
 \end{enumerate}
\end{cor}

\begin{proof}
The Corollary follows from Lemma \ref{lemma:deldelbar-lemma-a-doub}, Theorem \ref{thm:caratt-deldelbar-lemma-double}, and Lemma \ref{lemma:cohom-a-doub}.
\end{proof}

\begin{cor}\label{cor:charact-hodge-frolicher-double}
 Let $A^\bullet$ be a bounded $\Z$-graded $\mathbb{K}$-vector space endowed with two endomorphisms $\delta_1 \in \End^{\hat\delta_1}\left(A^\bullet\right)$ and $\delta_2 \in \End^{\hat\delta_2}\left(A^\bullet\right)$ such that $\delta_1^2 = \delta_2^2 = \delta_1\delta_2+\delta_2\delta_1 = 0$.
 Suppose that the greatest common divisor of $\hat\delta_1$ and $\hat\delta_2$ is $\GCD{\hat\delta_1}{\hat\delta_2}=1$, and that
 $$ \dim_\K H^{\bullet}_{\left( \delta_1 ; \delta_1 \right)}\left(A^{\bullet}\right) < +\infty  \qquad \text{ and } \qquad  \dim_\K H^{\bullet}_{\left( \delta_2 ; \delta_2 \right)}\left(A^{\bullet}\right) \;<\; +\infty \;.$$
 The following conditions are equivalent:
 \begin{enumerate}
  \item $A^{\bullet}$ satisfies the $\delta_1\delta_2$-Lemma;
  \item \begin{enumerate}[\itshape (a)]
        \item both the Hodge and Fr\"olicher spectral sequences of $\left(\Doub^{\bullet,\bullet}A^{\bullet} ,\, \delta_1\otimes_\K\id ,\, \delta_2\otimes_\K\beta\right)$ degenerate at the first level, equivalently, the equalities
        \begin{eqnarray*}
        \lefteqn{\dim_\K H^\bullet_{\left( \delta_1\otimes_\K\id + \delta_2\otimes_\K\beta ; \delta_1\otimes_\K\id + \delta_2\otimes_\K\beta \right)} \left(\Tot^\bullet\Doub^{\bullet,\bullet}A^\bullet\right)} \\[5pt]
        &=& \dim_\K \Tot^\bullet H^{\bullet,\bullet}_{\left( \delta_2\otimes_\K\beta ; \delta_2\otimes_\K\beta \right)}\left(\Doub^{\bullet,\bullet}A^{\bullet}\right) \\[5pt]
        &=& \dim_\K \Tot^\bullet H^{\bullet,\bullet}_{\left( \delta_1\otimes_\K\id ; \delta_1\otimes_\K\id \right)}\left(\Doub^{\bullet,\bullet}A^{\bullet}\right)
        \end{eqnarray*}
        hold;
        \item the equality
        \begin{eqnarray*}
        \lefteqn{\dim_\K H^{\bullet}_{\left( \delta_1 , \delta_2 ; \delta_1\delta_2 \right)}\left(A^{\bullet}\right) + \dim_\K H^{\bullet}_{\left( \delta_1\delta_2 ; \delta_1 , \delta_2 \right)}\left(A^{\bullet}\right)} \\[5pt]
        &=& \dim_\K H^{\bullet}_{\left( \delta_1 ; \delta_1 \right)}\left(A^{\bullet}\right) + \dim_\K H^{\bullet}_{\left( \delta_2 ; \delta_2 \right)}\left(A^{\bullet}\right)
        \end{eqnarray*}
        holds.
        \end{enumerate}
 \end{enumerate}
\end{cor}

\begin{proof}
The Corollary follows from Corollary \ref{cor:charact-hodge-frolicher-double-1}, Proposition \ref{prop:frolicher-classico}, Theorem \ref{thm:disug-frol}, and Lemma \ref{lemma:cohom-a-doub}.
\end{proof}

\section{Applications}

In this section, we prove or recover applications of the inequality {\itshape à la} Fr\"olicher, Theorem \ref{thm:disug-frol} and
Theorem \ref{thm:caratt-deldelbar-lemma-double}, to the complex, symplectic, and generalized complex cases.

\subsection{Complex structures}
Let $X$ be a compact complex manifold.
Consider the $\Z^2$-graded $\C$-vector space $\wedge^{\bullet,\bullet}X$ of bi-graded complex differential forms endowed with the endomorphisms $\del\in\End^{1,0}\left(\wedge^{\bullet,\bullet}X\right)$ and $\delbar\in\End^{0,1}\left(\wedge^{\bullet,\bullet}X\right)$, which satisfy $\del^2=\delbar^2=\del\delbar+\delbar\del=0$. As usual, define the Dolbeault cohomologies as
$$ H^{\bullet,\bullet}_{\del}(X) \;:=\; H^{\bullet,\bullet}_{\left( \del ; \del \right)}\left(\wedge^{\bullet,\bullet}X\right) \;, \qquad H^{\bullet,\bullet}_{\delbar}(X) \;:=\; H^{\bullet,\bullet}_{\left( \delbar ; \delbar \right)}\left(\wedge^{\bullet,\bullet}X\right) \;, $$
and the \emph{Bott-Chern cohomology} and the \emph{Aeppli cohomology} as, respectively, \cite{bott-chern, aeppli},
$$ H^{\bullet,\bullet}_{BC}(X) \;:=\; H^{\bullet,\bullet}_{\left( \del , \delbar ; \del\delbar \right)}\left(\wedge^{\bullet,\bullet}X\right) \;, \qquad H^{\bullet,\bullet}_{A}(X) \;:=\; H^{\bullet,\bullet}_{\left( \del\delbar ; \del , \delbar \right)}\left(\wedge^{\bullet,\bullet}X\right) \;. $$

Note that, since $X$ is a compact manifold, $\dim_\C \Tot^\bullet H^{\bullet,\bullet}_{\delbar}(X)<+\infty$: indeed, for any Hermitian metric $g$ with $\C$-linear Hodge-$*$-operator $*_g\colon \wedge^{\bullet_1,\bullet_2}X\to \wedge^{\dim_\C X - \bullet_2, \dim_\C X - \bullet_1}$, one has an isomorphism $\ker \left[\delbar,\, \delbar^*\right] \stackrel{\simeq}{\to} H^\bullet_{\delbar}(X)$, where $\delbar^*$ is the adjoint operator of $\delbar$ with respect to the inner product induced on $\wedge^{\bullet,\bullet}X$ by $g$, and the $2^{\text{nd}}$-order self-adjoint differential operator $\left[\delbar,\, \delbar^*\right]$ is elliptic.
Furthermore, $\dim_\C \Tot^\bullet H^{\bullet,\bullet}_{\del}(X)=\dim_\C \Tot^\bullet H^{\bullet,\bullet}_{\delbar}(X)<+\infty$, since conjugation induces the ($\C$-anti-linear) isomorphism $H^{\bullet_1,\bullet_2}_{\del}(X) \simeq H^{\bullet_2,\bullet_1}_{\delbar}(X)$ of $\R$-vector spaces.

Note also that $\dim_\C \Tot^\bullet H^{\bullet,\bullet}_{BC}(X) = \dim_\C \Tot^{2\,\dim_\C X - \bullet} H^{\bullet,\bullet}_{A}(X) < +\infty$, \cite[Corollaire 2.3, \S2.c]{schweitzer}: indeed, for any Hermitian metric $g$ on $X$, the $\C$-linear Hodge-$*$-operator $*_g\colon \wedge^{\bullet_1,\bullet_2}X\to \wedge^{\dim_\C X - \bullet_2, \dim_\C X - \bullet_1}X$ induces the isomorphism $*_g\colon H^{\bullet_1,\bullet_2}_{BC}(X) \stackrel{\simeq}{\to} H^{\dim_\C X - \bullet_2, \dim_\C X - \bullet_1}_{A}(X)$, \cite[\S2.c]{schweitzer}, and $\ker\tilde\Delta_{BC} \stackrel{\simeq}{\to} H^{\bullet,\bullet}_{BC}(X)$, \cite[Théorème 2.2]{schweitzer}, where $\tilde\Delta_{BC}:=\left(\del\delbar\right)\left(\del\delbar\right)^*+\left(\del\delbar\right)^*\left(\del\delbar\right)+\left(\delbar^*\del\right)\left(\delbar^*\del\right)^*+\left(\delbar^*\del\right)^*\left(\delbar^*\del\right)+\delbar^*\delbar+\del^*\del$ is a $4^{\text{th}}$-order self-adjoint elliptic differential operator,
\cite[Proposition 5]{kodaira-spencer-3}, see also \cite[\S2.b]{schweitzer}.

By abuse of notation, one says that $X$ \emph{satisfies the $\del\delbar$-Lemma} if the double complex $\left(\wedge^{\bullet,\bullet}X,\, \del,\, \delbar\right)$ satisfies the $\del\delbar$-Lemma, and one says that $X$ \emph{satisfies the $\de\de^{c}$-Lemma} if the $\Z$-graded $\C$-vector space $\wedge^\bullet X \otimes_\R\C$ endowed with the endomorphisms $\de\in\End^1\left(\wedge^\bullet X\otimes_\R\C\right)$ and $\de^c:=-\im\left(\del-\delbar\right)\in\End^1\left(\wedge^\bullet X\otimes_\R\C\right)$ such that $\de^2=\left(\de^c\right)^2=\left[\de,\, \de^c\right]=0$ satisfies the $\de\de^c$-Lemma. Actually, it turns out that $X$ satisfies the $\de\de^{c}$-Lemma if and only if $X$ satisfies the $\del\delbar$-Lemma, \cite[Remark 5.14]{deligne-griffiths-morgan-sullivan}: indeed, note that $\del = \frac{1}{2} \, \left(\de+\im\de^c\right)$ and $\delbar = \frac{1}{2}\, \left(\de-\im\de^c\right)$, and $\del\delbar = -\frac{\im}{2}\, \de\de^c$.

From Corollary \ref{cor:frolicher-like-double-complexes} and Theorem \ref{thm:caratt-deldelbar-lemma-double}, one gets straightforwardly the following inequality {\itshape à la} Fr\"olicher for the Bott-Chern cohomology of a compact complex manifolds and the corresponding characterization of the $\del\delbar$-Lemma by means of the Bott-Chern cohomology, first proved by the authors in \cite{angella-tomassini-3}.

\begin{cor}[{\cite[Theorem A, Theorem B]{angella-tomassini-3}}]\label{cor:cplx}
Let $X$ be a compact complex manifold. The inequality
\begin{equation}\label{eq:cplx}
\dim_\C \Tot^\bullet H^{\bullet,\bullet}_{BC}(X) + \dim_\C \Tot^\bullet H^{\bullet,\bullet}_{A}(X) \;\geq\; 2\, \dim_\C H^\bullet_{dR}(X;\C)
\end{equation}
holds.
Furthermore, the equality in \eqref{eq:cplx} holds if and only if $X$ satisfies the $\del\delbar$-Lemma.
\end{cor}

\subsection{Symplectic structures}
Let $X$ be a $2n$-dimensional compact manifold endowed with a \emph{symplectic structure} $\omega$, namely, a non-degenerate $\de$-closed $2$-form on $X$. The symplectic form $\omega$ induces a natural isomorphism $I \colon TX \stackrel{\simeq}{\to} T^*X$; more precisely, $I(\sspace)(\ssspace) := \omega(\sspace,\ssspace)$. Set $\Pi :=: \omega^{-1} := \omega\left(I^{-1}\sspace,I^{-1}\ssspace\right) \in \wedge^2 TX$ the \emph{canonical Poisson bi-vector} associated to $\omega$, namely, in a Darboux chart with local coordinates $\left\{x^1,\ldots,x^n,y^1,\ldots,y^n\right\}$ such that $\omega\stackrel{\text{loc}}{=}\sum_{j=1}^{n}\de x^j\wedge \de y^j$, one has $\omega^{-1}\stackrel{\text{loc}}{=}\sum_{j=1}^{n}\frac{\del}{\del x^j}\wedge\frac{\del}{\del y^j}$. One gets a bi-$\R$-linear form on $\wedge^k X$, denoted by $\left(\omega^{-1}\right)^k$, by defining it on the simple elements $\alpha^1\wedge\cdots \wedge \alpha^k \in \wedge^kX$ and $\beta^1\wedge\cdots \wedge \beta^k\in\wedge^kX$ as
$$
\left(\omega^{-1}\right)^k\left(\alpha^1\wedge\cdots \wedge \alpha^k,\, \beta^1\wedge\cdots \wedge \beta^k\right) \;:=\; \det\left(\omega^{-1}\left(\alpha^\ell,\, \beta^m\right)\right)_{\ell,m\in\{1,\ldots,k\}} \;;
$$
note that $\left(\omega^{-1}\right)^k$ is skew-symmetric, respectively symmetric, according to $k$ is odd, respectively even.

We recall that the operators
\begin{eqnarray*}
L \;\in\; \End^2\left(\wedge^\bullet X\right) \;, \quad && L(\alpha) \;:=\; \omega\wedge\alpha \;,\\[5pt]
\Lambda \;\in\; \End^{-2}\left(\wedge^\bullet X\right) \;, \quad && \Lambda(\alpha) \;:=\; -\iota_\Pi\alpha \;,\\[5pt]
H \;\in\; \End^0\left(\wedge^\bullet X\right) \;, \quad && H(\alpha) \;:=\; \sum_{k\in\Z} \left(n-k\right)\,\pi_{\wedge^kX}\alpha \;,
\end{eqnarray*}
yield an $\mathfrak{sl}(2;\R)$-representation on $\wedge^\bullet X$ (where $\iota_{\xi}\colon\wedge^{\bullet}X\to\wedge^{\bullet-2}X$ denotes the interior product with $\xi\in\wedge^2\left(TX\right)$, and $\pi_{\wedge^kX}\colon\wedge^\bullet X\to\wedge^kX$ denotes the natural projection onto $\wedge^kX$, for $k\in\Z$).

Define the \emph{symplectic co-differential operator} as
$$ \de^\Lambda \;:=\; \left[\de,\, \Lambda\right] \;\in\; \End^{-1}\left(\wedge^\bullet X \right) \;; $$
one has that $\left(\de^\Lambda\right)^2 = \left[\de,\, \de^\Lambda\right] = 0$, see \cite[page 266, page 265]{koszul}, \cite[Proposition 1.2.3, Theorem 1.3.1]{brylinski}.

\medskip

As a matter of notation, for $\sharp \in \left\{ \left( \de, \de^\Lambda ; \de\de^\Lambda \right) ,\, \left( \de ; \de \right) ,\, \left( \de^\Lambda ; \de^\Lambda \right) ,\, \left( \de\de^\Lambda ; \de , \de^\Lambda \right) \right\}$, we shorten $H^\bullet_{\sharp}(X):=H^\bullet_{\sharp}\left(\wedge^\bullet X\right)$. Note that $H^\bullet_{\left( \de ; \de \right)}(X)=H^\bullet_{dR}(X;\R)$. As regards notation introduced by L.-S. Tseng and S.-T. Yau in \cite[\S3]{tseng-yau-1}, note that $H^\bullet_{\left( \de^\Lambda ; \de^\Lambda\right)}(X) = H^\bullet_{\de^\Lambda}(X)$, and that $H^\bullet_{\left( \de , \de^\Lambda ; \de\de^\Lambda\right)}(X) = H^\bullet_{\de+\de^\Lambda}(X)$, and that $H^\bullet_{\left( \de\de^\Lambda ; \de , \de^\Lambda\right)}(X) = H^\bullet_{\de\de^\Lambda}(X)$.

Note also that, as a consequence of the Hodge theory developed by L.-S. Tseng and S.-T. Yau in \cite[Proposition 3.3, Theorem 3.5, Theorem 3.16]{tseng-yau-1}, one has that, \cite[Corollary 3.6, 
Corollary 3.17]{tseng-yau-1}, $X$ being compact, for $\sharp \in \left\{ \left( \de, \de^\Lambda ; \de\de^\Lambda \right) ,\, \left( \de ; \de \right) ,\, \left( \de^\Lambda ; \de^\Lambda \right) ,\, \left( \de\de^\Lambda ; \de , \de^\Lambda \right) \right\}$,
$$ \dim_\R H^\bullet_{\sharp}(X) \;<\; +\infty \;.$$

\medskip

With the aim to develop a symplectic counterpart of Riemannian Hodge theory for compact symplectic manifolds, J.-L. Brylinski defined the \emph{symplectic-$\star$-operator}, \cite[\S2]{brylinski},
$$
\star_\omega\colon \wedge^\bullet X\to \wedge^{2n-\bullet}X
$$
requiring that, for every $\alpha,\,\beta\in\wedge^k X$,
$$ \alpha\wedge\star_\omega \beta \;=\; \left(\omega^{-1}\right)^k\left(\alpha,\beta\right)\,\omega^n \;.$$

Since $\de^\Lambda\lfloor_{\wedge^kX}=(-1)^{k+1}\,\star_\omega\,\de\,\star_\omega$, \cite[Theorem 2.2.1]{brylinski}, and $\star_\omega^2=\id$, \cite[Lemma 2.1.2]{brylinski}, then one gets that $\star_\omega$ induces the isomorphism
$$ \star_\omega \colon H^\bullet_{\left( \de ; \de \right)}(X) \stackrel{\simeq}{\to} H^{2n-\bullet}_{\left( \de^\Lambda ; \de^\Lambda \right)}(X) \;.$$
In particular, by the Poincaré duality, it follows that
$$ \dim_\R H^\bullet_{\left( \de ; \de \right)}(X) \;=\; \dim_\R H^\bullet_{\left( \de^\Lambda ; \de^\Lambda \right)}(X) \;<\; +\infty \;.$$

Furthermore, by choosing an almost-complex structure $J$ compatible with $\omega$ (namely, such that $\omega(\sspace, \, J\sspace)$ is positive definite and $\omega(J\sspace,\, J\ssspace)=\omega$), and by considering the $J$-Hermitian metric $g:=\omega(\sspace,\, J\ssspace)$, one gets that, \cite[Corollary 3.25]{tseng-yau-1}, the Hodge-$*$-operator $*_g\colon \wedge^\bullet X \to \wedge^{2n-\bullet}X$ associated to $g$ induces the isomorphism, \cite[Corollary 2.2.2]{brylinski},
$$ *_g \colon H^\bullet_{\left( \de , \de^\Lambda ; \de\de^\Lambda \right)}(X) \stackrel{\simeq}{\to} H^{2n-\bullet}_{\left( \de\de^\Lambda ; \de , \de^\Lambda \right)}(X) \;.$$
In particular, it follows that
$$ \dim_\R H^\bullet_{\left( \de , \de^\Lambda ; \de\de^\Lambda \right)}(X) \;=\; \dim_\R H^{2n-\bullet}_{\left( \de\de^\Lambda ; \de , \de^\Lambda \right)}(X) \;<\; +\infty \;.$$

\medskip

Recall that one says that the \emph{Hard Lefschetz Condition} holds on $X$ if
\begin{equation}\label{eq:hlc}
\tag{HLC}
 \text{for every } k\in\N\;, \qquad L^k\colon H^{n-k}_{dR}(X;\R) \stackrel{\simeq}{\to} H^{n+k}_{dR}(X;\R) \;.
\end{equation}

\medskip

As in \cite{angella-tomassini-4}, and miming \cite{li-zhang} in the almost-complex case, define, for $r,s\in\N$,
$$ H^{(r,s)}_\omega(X;\R) \;:=\; \left\{\left[L^r\,\gamma^{(s)}\right]\in H^{2r+s}_{dR}(X;\R) \st \Lambda \gamma^{(s)} = 0 \right\} \;\subseteq\; H^{2r+s}_{dR}(X;\R) \;; $$
one has that
$$ \sum_{2r+s=\bullet} H^{(r,s)}_\omega(X;\R) \;\subseteq\; H^{\bullet}_{dR}(X;\R) \;,$$
but in general neither the sum is direct, nor the inclusion is an equality.

As proved by Y. Lin in \cite[Proposition A.5]{lin}, if the Hard Lefschetz Condition holds on $X$, then
$$ H^{(0,\bullet)}_\omega(X;\R) \;=\; \mathrm{P}H^\bullet_{\de}(X;\R) \;, $$
where
$$ \mathrm{P}H^\bullet_{\de}(X;\R) \;:=\; \frac{\ker\de\cap\ker\de^\Lambda\cap\ker\Lambda}{\imm\de\lfloor_{\ker\de^\Lambda\cap\ker\Lambda}}  $$
is the \emph{primitive cohomology} introduced by L.-S. Tseng and S.-T. Yau in \cite[\S4.1]{tseng-yau-1}.

\medskip

We recall the following result.

\begin{thm}[{\cite[Corollary 2]{mathieu}, \cite[Theorem 0.1]{yan}, \cite[Proposition 1.4]{merkulov}, \cite{guillemin}, \cite[Proposition 3.13]{tseng-yau-1}, \cite[Theorem 5.4]{cavalcanti}, \cite[Remark 2.3]{angella-tomassini-4}}]
 Let $X$ be a compact manifold endowed with a symplectic structure $\omega$. The following conditions are equivalent:
 \begin{enumerate}
  \item every de Rham cohomology class of $X$ admits a representative being both $\de$-closed and $\de^\Lambda$-closed, namely, Brylinski's conjecture \cite[Conjecture 2.2.7]{brylinski} holds on $X$;
  \item the Hard Lefschetz Condition holds on $X$;
  \item the natural map $H_{\left(\de , \de^\Lambda ; \de\de^\Lambda\right)}^\bullet\left(\wedge^\bullet X\right) \to H_{dR}^{\bullet}(X;\R)$ induced by the identity is surjective;
  \item the natural map $H_{\left(\de , \de^\Lambda ; \de\de^\Lambda\right)}^\bullet\left(\wedge^\bullet X\right) \to H_{dR}^{\bullet}(X;\R)$ induced by the identity is an isomorphism;
  \item the bounded $\Z$-graded $\R$-vector space $\wedge^\bullet X$ endowed with the endomorphisms $\de\in\End^1\left(\wedge^\bullet X\right)$ and $\de^\Lambda \in \End^{-1}\left(\wedge^\bullet X \right)$ satisfies the $\de\de^\Lambda$-Lemma;
  \item the decomposition
    $$ H^\bullet_{dR}(X;\R) \;=\; \bigoplus_{r\in\N} L^r\, H^{(0,\bullet-2r)}_\omega(X;\R) \;, $$
    holds.
 \end{enumerate}
\end{thm}

\medskip

In order to apply Corollary \ref{cor:charact-hodge-frolicher-double} to the $\Z$-graded $\R$-vector space $\wedge^\bullet X$ endowed with the endomorphisms $\de \in \End^1\left(\wedge^\bullet X\right)$ and $\de^\Lambda \in \End^{-1}\left(\wedge^\bullet X \right)$, satisfying $\de^2 = \left(\de^\Lambda\right)^2 = \left[\de,\, \de^\Lambda\right] = 0$, we need the following result.

\begin{lemma}[{\cite[Theorem 2.3.1]{brylinski}, \cite[Theorem 2.5]{fernandez-ibanez-deleon-IsrJMath}; see also \cite[Theorem 2.9]{fernandez-ibanez-deleon-IsrJMath}, \cite[Theorem 5.2]{cavalcanti-jgp}}]
 Let $X$ be a compact manifold endowed with a symplectic structure $\omega$. Consider the $\Z^2$-graded $\R$-vector space $\wedge^{\bullet}X$ endowed with the endomorphisms $\de \in \End^1\left(\wedge^\bullet X\right)$ and $\de^\Lambda \in \End^{-1}\left(\wedge^\bullet X\right)$. Both the spectral sequences associated to the canonical double complex $\left(\Doub^{\bullet,\bullet}\wedge^{\bullet}X ,\, \de\otimes_\R\id ,\, \de^\Lambda\otimes_\R\beta\right)$ degenerate at the first level.
\end{lemma}

Hence, by applying Theorem \ref{thm:disug-frol} and Corollary \ref{cor:charact-hodge-frolicher-double} to the $\Z$-graded $\R$-vector space $\wedge^\bullet X$ endowed with $\de\in\End^{1}\left(\wedge^\bullet X\right)$ and $\de^\Lambda\in\End^{-1}\left(\wedge^\bullet X\right)$, we get the following result.

\begin{thm}\label{thm:sympl}
Let $X$ be a compact manifold endowed with a symplectic structure $\omega$.
The inequality
\begin{equation}\label{eq:sympl}
\dim_\R H^{\bullet}_{\left( \de , \de^\Lambda ; \de\de^\Lambda \right)}\left(X\right) + \dim_\R H^{\bullet}_{\left( \de\de^\Lambda ; \de , \de^\Lambda \right)}\left(X\right) \;\geq\; 2\, \dim_\R H^{\bullet}_{dR}(X;\R)
\end{equation}
holds. Furthermore, the equality in \eqref{eq:sympl} holds if and only if $X$ satisfies the Hard Lefschetz Condition.
\end{thm}

\medskip

Consider $X = \left. \Gamma \middle\backslash G \right.$ a solvmanifold endowed with a $G$-left-invariant symplectic structure $\omega$; in particular, $\omega$ induces a linear symplectic structure on $\mathfrak{g}$; therefore the endomorphisms $\de \in \End^1\left(\wedge^\bullet X\right)$ and $\de^\Lambda \in \End^{-1}\left(\wedge^\bullet X \right)$ yield endomorphisms $\de \in \End^1\left(\wedge^\bullet \duale{\g}\right)$ and $\de^\Lambda \in \End^{-1}\left(\wedge^\bullet \duale{\g} \right)$ on the $\Z$-graded $\R$-vector sub-space $\wedge^\bullet\duale{\g} \hookrightarrow \wedge^\bullet X$, where we identify objects on $\mathfrak{g}$ with $G$-left-invariant objects on $X$ by means of left-translations. For $\sharp \in \left\{ \left( \de, \de^\Lambda ; \de\de^\Lambda \right) ,\, \left( \de ; \de \right) ,\, \left( \de^\Lambda ; \de^\Lambda \right) ,\, \left( \de\de^\Lambda ; \de , \de^\Lambda \right) \right\}$, one has the natural map $\iota\colon H^\bullet_{\sharp}\left(\wedge^\bullet\mathfrak{g}^*\right)
 \to H^\bullet_{\sharp}\left(X\right)$.
We recall the following result, which allows to compute the cohomologies of a completely-solvable solvmanifold by using just left-invariant forms; recall, e.g., that, by A. Hattori's theorem \cite[Corollary 4.2]{hattori}, if $G$ is \emph{completely-solvable} (that is, for any $g\in G$, all the eigenvalues of $\mathrm{Ad}_g:=\de\left(\psi_g\right)_e\in\mathrm{Aut}(\g)$ are real, equivalently, if, for any $X\in\g$, all the eigenvalues of $\mathrm{ad}_X:=\left[X,\sspace\right]\in\End(\g)$ are real, where $\psi\colon G \ni g \mapsto \left( \psi_g \colon h \mapsto g\,h\,g^{-1} \right) \in \mathrm{Aut}(G)$ and $e$ is the identity element of $G$), then the natural map $H^\bullet_{dR}\left(\wedge^\bullet\duale{\g}\right) \to H^\bullet_{dR}\left(X;\R\right)$ is an isomorphism.

\begin{thm}[{\cite[Theorem 3, Remark 4]{macri}, see also \cite{angella-kasuya}}]
 Let $X = \left. \Gamma \middle\backslash G \right.$ be a completely-solvable solvmanifold endowed with a $G$-left-invariant symplectic structure $\omega$. Then, for $\sharp \in \left\{ \left( \de, \de^\Lambda ; \de\de^\Lambda \right) ,\, \left( \de ; \de \right) ,\, \left( \de^\Lambda ; \de^\Lambda \right) ,\, \left( \de\de^\Lambda ; \de , \de^\Lambda \right) \right\}$, the natural map
 $$ \iota\colon H^\bullet_{\sharp}\left(\wedge^\bullet\mathfrak{g}^*\right) \to H^\bullet_{\sharp}\left(X\right) $$
 is an isomorphism.
\end{thm}

\begin{ex}
Let $\mathbb{I}_3:=\left. \Z\left[\im\right]^3 \middle\backslash \left(\C^3,\,*\right) \right.$ be the {\em Iwasawa manifold}, where the group structure $*$ on $\C^3$ is defined by
$$ \left(z_1,\,z_2,\,z_3\right) * \left(w_1,\,w_2,\,w_3\right) := \left(z_1+w_1 ,\, z_2+w_2 ,\, z_3+z_1w_2+w_3\right) \;.$$
There exists a $\left(\C^3,\,*\right)$-left-invariant co-frame $\left\{e^j\right\}_{j\in\{1,\ldots,6\}}$ of $T^*X$ such that
$$
\de e^1 \;=\; \de e^2 \;=\; \de e^3 \;=\; \de e^4 \;=\; 0 \;, \qquad \de e^5 \;=\; -e^{13}+e^{24} \;, \qquad \de e^6 \;=\; -e^{14}-e^{23}
$$
(in order to simplify notation, we shorten, e.g., $e^{12} := e^1\wedge e^2$).

Consider the $\left(\C^3,\,*\right)$-left-invariant almost-K\"ahler structure $\left(J,\, \omega,\, g\right)$ on $\mathbb{I}_3$ defined by
$$ Je^1 \;:=\; -e^6\;,\quad  Je^2 \;:=\; -e^5\;,\quad Je^3 \;:=\; -e^4 \;, \qquad \omega \;:=\; e^{16}+e^{25}+e^{34}\;, \qquad g \;:=\; \omega\left(\sspace,\, J\ssspace\right) \;; $$
it has been studied in \cite[\S4]{angella-tomassini-zhang} as an example of an almost-K\"ahler structure non-inducing a decomposition in cohomology according to the almost-complex structure, \cite[Proposition 4.1]{angella-tomassini-zhang}.

The symplectic cohomologies of the Iwasawa manifold $\mathbb{I}_3$ endowed with the $\left(\C^3,\,*\right)$-left-invariant symplectic structure $\omega$ can be computed using just $\left(\C^3,\,*\right)$-left-invariant forms, and their real dimensions are summarized in Table \ref{table:iwasawa-alm-kahler}.

\begin{center}
\begin{table}[h]
 \centering
\begin{tabular}{>{$\mathbf\bgroup}c<{\mathbf\egroup$} || >{$}c<{$} | >{$}c<{$} | >{$}c<{$} | >{$}c<{$}}
\toprule
\dim_\C H_{\sharp}^{\bullet}\left(\mathbb{I}_3\right) & \left( \de ; \de \right) & \left( \de^\Lambda ; \de^\Lambda \right) & \left( \de , \de^\Lambda ; \de\de^\Lambda \right) & \left( \de\de^\Lambda ; \de , \de^\Lambda \right) \\
\toprule
0 & 1 &  1 & 1 & 1 \\
\midrule[0.02em]
1 & 4 & 4 & 4 & 4 \\
\midrule[0.02em]
2 & 8 & 8 & 9 & 10 \\
\midrule[0.02em]
3 & 10 & 10 & 11 & 11 \\
\midrule[0.02em]
4 & 8 & 8 & 10 & 9 \\
\midrule[0.02em]
5 & 4 & 4 & 4 & 4 \\
\midrule[0.02em]
6 & 1 & 1 & 1 & 1 \\
\bottomrule
\end{tabular}
\caption{The symplectic cohomologies of the Iwasawa manifold $\mathbb{I}_3:=\left. \Z\left[\im\right]^3 \middle\backslash \left(\C^3,\,*\right) \right.$ endowed with the symplectic structure $\omega := e^{1} \wedge e^{6} + e^{2} \wedge e^{5} + e^{3} \wedge e^{4}$.}
\label{table:iwasawa-alm-kahler}
\end{table}
\end{center}

In particular, note that
\begin{eqnarray*}
\dim_\K H^{1}_{\left( \de , \de^\Lambda ; \de\de^\Lambda \right)}\left(X\right) + \dim_\K H^{1}_{\left( \de\de^\Lambda ; \de , \de^\Lambda \right)}\left(X\right) - 2\, \dim_\K H^{1}_{dR}(X;\R) &=& 0  \;, \\[5pt]
\dim_\K H^{2}_{\left( \de , \de^\Lambda ; \de\de^\Lambda \right)}\left(X\right) + \dim_\K H^{2}_{\left( \de\de^\Lambda ; \de , \de^\Lambda \right)}\left(X\right) - 2\, \dim_\K H^{2}_{dR}(X;\R) &=& 3  \;, \\[5pt]
\dim_\K H^{3}_{\left( \de , \de^\Lambda ; \de\de^\Lambda \right)}\left(X\right) + \dim_\K H^{3}_{\left( \de\de^\Lambda ; \de , \de^\Lambda \right)}\left(X\right) - 2\, \dim_\K H^{3}_{dR}(X;\R) &=& 2  \;.
\end{eqnarray*}
\end{ex}

\begin{rem}
 More in general, let $X$ be a compact manifold endowed with a Poisson bracket $\left\{\sspace, \ssspace\right\}$, and denote by $G$ the Poisson tensor associated to $\left\{\sspace, \ssspace\right\}$.
 By following J.-L. Koszul, \cite{koszul}, one defines $\delta:=\left[\iota_G,\, \de\right]\in\End^{-1}\left(\wedge^\bullet X\right)$. One has that $\delta^2=0$ and $\left[\de,\delta\right]=0$, \cite[page 266, page 265]{koszul}, see also \cite[Proposition 1.2.3, Theorem 1.3.1]{brylinski}.
 
 One has that, on any compact Poisson manifold, the first spectral sequence ${'E}_r^{\bullet,\bullet}$ associated to the canonical double complex $\left(\Doub^{\bullet,\bullet}\wedge^\bullet X,\, \de\otimes_\R\id,\, \delta\otimes_\R\beta\right)$ degenerates at the first level, \cite[Theorem 2.5]{fernandez-ibanez-deleon-IsrJMath}.

 On the other hand, M. Fernández, R. Ibáñez, and M. de León provided an example of a compact Poisson manifold (more precisely, of a nilmanifold endowed with a co-symplectic structure) such that the second spectral sequence ${''E}_r^{\bullet,\bullet}\left(\Doub^{\bullet,\bullet}\wedge^\bullet X,\, \de\otimes_\R\id,\, \delta\otimes_\R\beta\right)$ does not degenerate at the first level, \cite[Theorem 5.1]{fernandez-ibanez-deleon-IsrJMath}.

 In fact, on a compact $2n$-dimensional manifold $X$ endowed with a symplectic structure $\omega$, the symplectic-$\star$-operator $\star_\omega\colon \wedge^\bullet X \to \wedge^{2n-\bullet}X$ induces the isomorphism $\star_\omega\colon {'E}^{\bullet_1,\bullet_2}_r \stackrel{\simeq}{\to} {''E}^{\bullet_2,2n+\bullet_1}_r$, \cite[Theorem 2.9]{fernandez-ibanez-deleon-IsrJMath}; it follows that, on a compact symplectic manifold, also the second spectral sequence ${''E}_r^{\bullet,\bullet}\left(\Doub^{\bullet,\bullet}\wedge^\bullet X,\, \de\otimes_\R\id,\, \delta\otimes_\R\beta\right)$ actually degenerates at the first level, \cite[Theorem 2.3.1]{brylinski}, see also \cite[Theorem 2.8]{fernandez-ibanez-deleon-IsrJMath}.
\end{rem}

\subsection{Generalized complex structures}\label{subsec:gen-cplx}

Let $X$ be a compact differentiable manifold of dimension $2n$.
Consider the bundle $TX\oplus T^*X$ endowed with the natural symmetric pairing
$$ \scalar{\sspace}{\ssspace} \colon \left(TX\oplus T^*X\right) \times \left(TX\oplus T^*X\right) \to \R \;, \qquad \scalar{X+\xi}{Y+\eta}\;:=\; \frac{1}{2}\,\left(\xi(Y)+\eta(X)\right) \;. $$
Fix a $\de$-closed $3$-form $H$ on $X$.
On the space $\mathcal{C}^{\infty}\left(X; \, TX\oplus T^*X\right)$ of smooth sections of $TX\oplus T^*X$ over $X$, define the \emph{$H$-twisted Courant bracket} as
\begin{eqnarray*}
&& \left[\sspace ,\, \ssspace\right]_H \colon \mathcal{C}^{\infty}\left(X; \, TX\oplus T^*X\right) \times \mathcal{C}^{\infty}\left(X; \, TX\oplus T^*X\right) \to \mathcal{C}^{\infty}\left(X; \, TX\oplus T^*X\right) \;, \\[5pt]
&& \left[X+\xi,\, Y+\eta\right]_H \;:=\; \left[X,\, Y\right] + \mathcal{L}_X\eta - \mathcal{L}_Y\xi - \frac{1}{2}\, \de \left(\iota_X\eta-\iota_Y\xi\right) + \iota_Y\iota_X H
\end{eqnarray*}
(where $\iota_{X}\in \End^{-1}\left(\wedge^\bullet X\right)$ denotes the interior product with $X\in \mathcal{C}^\infty(X;TX)$ and $\mathcal{L}_X:=\left[\iota_X,\, \de\right]\in \End^0\left(\wedge^\bullet X\right)$ denotes the Lie derivative along $X\in \mathcal{C}^\infty(X;TX)$); the $H$-twisted Courant bracket can be seen also as a derived bracket induced by the $H$-twisted differential $\de_H:=\de+H\wedge\sspace$,
see \cite[\S3.2]{gualtieri-phdthesis}, \cite[\S2]{gualtieri}.

Furthermore, consider the \emph{Clifford action} of $TX\oplus T^*X$ on the space of differential forms with respect to $\scalar{\sspace}{\ssspace}$,
$$ \Cliff\left(TX\oplus T^*X\right) \times \wedge^\bullet X \to \wedge^{\bullet-1} X \oplus \wedge^{\bullet+1} X \;, \qquad (X+\xi) \cdot \varphi \;:=\; \iota_X\varphi + \xi\wedge\varphi \;, $$
and its bi-$\C$-linear extension $\Cliff\left(\left(TX\oplus T^*X\right)\otimes_\R\C\right) \times \left(\wedge^\bullet X\otimes_\R\C\right) \to \left(\wedge^{\bullet-1} X\otimes_\R\C\right) \oplus \left(\wedge^{\bullet+1} X\otimes_\R\C\right)$, where $\Cliff\left(TX\oplus T^*X\right) := \left.\left(\bigoplus_{k\in\Z}\bigotimes_{j=1}^{k}\left(TX\oplus T^*X\right)\right)\middle\slash\left\{v\otimes_\R v - \scalar{v}{v} \st v\in TX\oplus T^*X\right\} \right.$ is the Clifford algebra associated to $TX\oplus T^*X$ and $\scalar{\sspace}{\ssspace}$.

\medskip

Recall that an \emph{$H$-twisted generalized complex structure} on $X$, \cite[Definition 4.14, Definition 4.18]{gualtieri-phdthesis}, \cite[Definition 3.1]{gualtieri} is an endomorphism $\mathcal{J}\in\End\left(TX\oplus T^*X\right)$ such that
\begin{inparaenum}[\itshape (i)]
 \item $\mathcal{J}^2 = -\id_{TX\oplus T^*X}$, and
 \item $\mathcal{J}$ is orthogonal with respect to $\scalar{\sspace}{\ssspace}$, and
 \item the Nijenhuis tensor
       $$ \mathrm{Nij}_{\mathcal{J},H} \;:=\; -\left[ \mathcal{J}\,\sspace ,\, \mathcal{J}\,\ssspace\right]_H + \mathcal{J} \left[ \mathcal{J}\,\sspace ,\, \ssspace \right]_H + \mathcal{J} \left[ \sspace ,\, \mathcal{J}\,\ssspace \right]_H + \mathcal{J} \left[ \sspace ,\, \ssspace \right]_H \;\in\; \left(TX\oplus T^*X\right) \otimes_\R \left(TX\oplus T^*X\right) \otimes_\R \left(TX\oplus T^*X\right)^* $$
       of $\mathcal{J}$ with respect to the $H$-twisted Courant bracket vanishes identically.
\end{inparaenum}

Equivalently, \cite[Proposition 4.3]{gualtieri-phdthesis}, (by setting $L:=:L_{\mathcal{J}}$ the $\im$-eigen-bundle of the $\C$-linear extension of $\mathcal{J}$ to $\left(TX \oplus T^*X\right) \otimes_\R \C$), a generalized complex structure on $X$ is identified by a sub-bundle $L$ of $\left(TX \oplus T^*X\right) \otimes_\R \C$ such that
\begin{inparaenum}[\itshape (i)]
 \item $L$ is maximal isotropic with respect to $\scalar{\sspace}{\ssspace}$, and
 \item $L$ is involutive with respect to the $H$-twisted Courant bracket, and
 \item $L \cap \bar{L} = \{0\}$.
\end{inparaenum}

Equivalently, \cite[Theorem 4.8]{gualtieri-phdthesis}, (by choosing a complex form $\rho$ whose Clifford annihilator
$$ L_\rho \;:=\; \left\{v\in \left(TX \oplus T^*X\right) \otimes_\R \C \st v\cdot \rho=0\right\} $$
is the $\im$-eigen-bundle $L_\mathcal{J}$ of the $\C$-linear extension of $\mathcal{J}$ to $\left(TX \oplus T^*X\right) \otimes_\R \C$), a generalized complex structure on $X$ is identified by a sub-bundle $U:=:U_{\mathcal{J}}$ (which is called the \emph{canonical bundle}, \cite[\S4.1]{gualtieri-phdthesis}, \cite[Definition 3.7]{gualtieri}) of complex rank $1$ of $\wedge^\bullet X \otimes_\R \C$ being locally generated by a form $\rho = \exp{\left( B+\im\omega \right)}\wedge\Omega$, where $B\in\wedge^2X$, and $\omega\in\wedge^2X$, and $\Omega=\theta^1\wedge\cdots\wedge\theta^k\in\wedge^kX\otimes_\R\C$ with $\left\{\theta^1,\ldots,\theta^k\right\}$ a set of linearly independent complex $1$-forms, such that
\begin{inparaenum}[\itshape (i)]
 \item $\Omega \wedge \bar\Omega \wedge \omega^{n-k} \neq 0$, and
 \item there exists $v\in\left(TX \oplus T^*X\right) \otimes_\R \C$ such that $\de_H\rho = v \cdot \rho$, where $\de_H:=\de+H\wedge\sspace$.
\end{inparaenum}

By definition, the \emph{type} of a generalized complex structure $\mathcal{J}$ on $X$, \cite[\S4.3]{gualtieri-phdthesis}, \cite[Definition 3.5]{gualtieri}, is the upper-semi-continuous function
$$ \mathrm{type}\left(\mathcal{J}\right) \;:=\; \frac{1}{2}\, \dim_\R \left( T^*X \cap \mathcal{J}T^*X \right) $$
on $X$, equivalently, \cite[Definition 1.1]{gualtieri}, the degree of the form $\Omega$.

\medskip

A generalized complex structure $\mathcal{J}$ on $X$ induces a $\Z$-graduation on the space of complex differential forms on $X$, \cite[\S4.4]{gualtieri-phdthesis}, \cite[Proposition 3.8]{gualtieri}. Namely, define, for $k\in\Z$,
$$ U^k \;:=:\; U^k_{\mathcal{J}} \;:=\; \wedge^{n-k}\bar L_{\mathcal{J}} \cdot U_{\mathcal{J}} \;\subseteq\; \wedge^\bullet X \otimes_\R \C \;, $$
where $L_{\mathcal{J}}$ is the $\im$-eigenspace of the $\C$-linear extension of $\mathcal{J}$ to $\left(TX \oplus T^*X\right)\otimes_\R \C$ and $U^n_{\mathcal{J}}:=U_{\mathcal{J}}$ is the canonical bundle of $\mathcal{J}$.

\medskip

For a $\scalar{\sspace}{\ssspace}$-orthogonal endomorphism $\mathcal{J}\in\End\left(TX\oplus T^*X\right)$ satisfying $\mathcal{J}^2 = -\id_{TX\oplus T^*X}$,
the $\Z$-graduation $U^\bullet_{\mathcal{J}}$ still makes sense,
and the condition that $\mathrm{Nij}_{\mathcal{J},H}=0$ turns out to be equivalent, \cite[Theorem 4.3]{gualtieri-phdthesis}, \cite[Theorem 3.14]{gualtieri}, to
$$ \de_H\colon U^\bullet_{\mathcal{J}} \to U^{\bullet+1}_{\mathcal{J}} \oplus U^{\bullet-1}_{\mathcal{J}} \;. $$

Therefore, on a compact differentiable manifold endowed with a generalized complex structure $\mathcal{J}$, one has, \cite[\S4.4]{gualtieri-phdthesis}, \cite[\S3]{gualtieri},
$$ \de_H \;=\; \del_{\mathcal{J},H} + \delbar_{\mathcal{J},H} \qquad \text{ where } \qquad \del_{\mathcal{J},H} \colon U^\bullet_{\mathcal{J}} \to U^{\bullet+1}_{\mathcal{J}} \quad \text{ and } \quad \delbar_{\mathcal{J},H} \colon U^\bullet_{\mathcal{J}} \to U^{\bullet-1}_{\mathcal{J}} \;.$$
Define also, \cite[page 52]{gualtieri-phdthesis}, \cite[Remark at page 97]{gualtieri},
$$ \de_H^{\mathcal{J}} \;:=\; -\im\left(\del_{\mathcal{J},H} - \delbar_{\mathcal{J},H}\right) \colon U^\bullet_{\mathcal{J}} \to U^{\bullet+1}_{\mathcal{J}} \oplus U^{\bullet-1}_{\mathcal{J}} \;.$$

\medskip

By abuse of notation, one says that $X$ \emph{satisfies the $\del_{\mathcal{J},H}\delbar_{\mathcal{J},H}$-Lemma} if $\left(U^\bullet,\, \del_{\mathcal{J},H},\, \delbar_{\mathcal{J},H}\right)$ satisfies the $\del_{\mathcal{J},H}\delbar_{\mathcal{J},H}$-Lemma, and one says that $X$ \emph{satisfies the $\de_H\de_H^{\mathcal{J}}$-Lemma} if $\left(U^\bullet,\, \de_H,\, \de_H^{\mathcal{J}}\right)$ satisfies the $\de_H\de_H^{\mathcal{J}}$-Lemma. Actually, it turns out that $X$ satisfies the $\de_H\de_H^{\mathcal{J}}$-Lemma if and only if $X$ satisfies the $\del_{\mathcal{J},H}\delbar_{\mathcal{J},H}$-Lemma, \cite[Remark at page 129]{cavalcanti-jgp}: indeed, note that $\ker\del_{\mathcal{J},H}\delbar_{\mathcal{J},H} = \ker \de_H\de_H^{\mathcal{J}}$, and $\ker\del_{\mathcal{J},H} \cap \ker\delbar_{\mathcal{J},H} = \ker\de_H \cap \ker\de_H^{\mathcal{J}}$, and $\imm\del_{\mathcal{J},H} + \imm\delbar_{\mathcal{J},H} = \imm\de_H + \imm\de_H^{\mathcal{J}}$.

Moreover, the following result by G.~R. Cavalcanti holds.

\begin{thm}[{\cite[Theorem 4.2]{cavalcanti}, \cite[Theorem 4.1, Corollary 2]{cavalcanti-jgp}}]
 A manifold $X$ endowed with an $H$-twisted generalized complex structure $\mathcal{J}$ satisfies the $\de_H\de_H^{\mathcal{J}}$-Lemma if and only if $\left(\ker\de_H^{\mathcal{J}},\, \de\right)\hookrightarrow\left(U^\bullet,\, \de_H\right)$ is a quasi-isomorphism of differential $\Z$-graded $\C$-vector spaces. In this case, it follows that the splitting $\wedge^\bullet X \otimes_\R \C = \bigoplus_{k\in\Z}U^k$ gives rise to a decomposition in cohomology.
\end{thm}

An application of \cite[Proposition 5.17, 5.21]{deligne-griffiths-morgan-sullivan} yields the following result.
\begin{thm}[{\cite[Theorem 4.4]{cavalcanti}, \cite[Theorem 5.1]{cavalcanti-jgp}}]
 A manifold $X$ endowed with an $H$-twisted generalized complex structure $\mathcal{J}$ satisfies the $\de\de^{\mathcal{J}}$-lemma if and only if the canonical spectral sequence degenerates at the first level and the decomposition of complex forms into sub-bundles $U^k$ induces a decomposition in cohomology.
\end{thm}

\medskip

Given a compact complex manifold $X$ endowed with an $H$-twisted generalized complex structure, consider the following cohomologies:
$$ GH_{dR_{H}}(X) \;:=\; H_{\left( \de_H ; \de_H \right)}\left(\Tot U^\bullet_{\mathcal{J}}\right) \;, $$
and
$$ GH^{\bullet}_{\delbar_{\mathcal{J},H}}(X) \;:=\; H^\bullet_{\left( \delbar_{\mathcal{J},H} ; \delbar_{\mathcal{J},H} \right)}\left(U^\bullet_{\mathcal{J}}\right) \;, \quad GH^{\bullet}_{\del_{\mathcal{J},H}}(X) \;:=\; H^\bullet_{\left( \del_{\mathcal{J},H} ; \del_{\mathcal{J},H} \right)}\left(U^\bullet_{\mathcal{J}}\right) \;, $$
and
$$ GH^{\bullet}_{BC_{\mathcal{J},H}}(X) \;:=\; H^\bullet_{\left( \del_{\mathcal{J},H} , \delbar_{\mathcal{J},H} ; \del_{\mathcal{J},H}\delbar_{\mathcal{J},H} \right)}\left(U^\bullet_{\mathcal{J}}\right) \;, \qquad GH^{\bullet}_{A_{\mathcal{J},H}}(X) \;:=\; H^\bullet_{\left( \del_{\mathcal{J},H}\delbar_{\mathcal{J},H} ; \del_{\mathcal{J},H} , \delbar_{\mathcal{J},H} \right)}\left(U^\bullet_{\mathcal{J}}\right) \;. $$
Note that, for $H=0$, one has $GH_{dR_{0}}(X)=\Tot H^\bullet_{dR}(X;\R)$.

By \cite[Proposition 5.1]{gualtieri-phdthesis}, \cite[Proposition 3.15]{gualtieri}, it follows that $\dim_\C GH^{\bullet}_{\del_{\mathcal{J},H}}(X)<+\infty$ and $\dim_\C GH^{\bullet}_{\delbar_{\mathcal{J},H}}(X)<+\infty$.

As an application of Theorem \ref{thm:disug-frol}, we get the following result.

\begin{thm}\label{thm:gen-frol-ineq}
 Let $X$ be a compact differentiable manifold endowed with an $H$-twisted generalized complex structure $\mathcal{J}$. Then
 \begin{equation}\label{eq:ineq-frol-cplx-gen}
 \dim_\C GH^{\bullet}_{BC_{\mathcal{J},H}}(X) + \dim_\C GH^{\bullet}_{A_{\mathcal{J},H}}(X) \;\geq\; \dim_\C GH^{\bullet}_{\delbar_{\mathcal{J},H}}(X) + \dim_\C GH^{\bullet}_{\del_{\mathcal{J},H}}(X) \;.
 \end{equation}
\end{thm}

As an application of Corollary \ref{cor:charact-hodge-frolicher-double}, we get the following result; compare it also with \cite[Theorem 4.4]{cavalcanti}.

\begin{thm}\label{thm:gen-charact}
 Let $X$ be a compact differentiable manifold endowed with an $H$-twisted generalized complex structure $\mathcal{J}$.
 The following conditions are equivalent:
 \begin{itemize}
  \item $X$ satisfies the $\del_{\mathcal{J},H}\delbar_{\mathcal{J},H}$-Lemma;
  \item the Hodge and Fr\"olicher spectral sequences associated to the canonical double complex $\left(\Doub^{\bullet,\bullet}U^\bullet_{\mathcal{J}},\, \del_{\mathcal{J},H} \otimes_\C \id,\, \delbar_{\mathcal{J},H} \otimes_\C \beta\right)$ degenerate at the first level and the equality in \eqref{eq:ineq-frol-cplx-gen},
  $$ \dim_\C GH^{\bullet}_{BC_{\mathcal{J},H}}(X) + \dim_\C GH^{\bullet}_{A_{\mathcal{J},H}}(X) \;=\; \dim_\C GH^{\bullet}_{\delbar_{\mathcal{J},H}}(X) + \dim_\C GH^{\bullet}_{\del_{\mathcal{J},H}}(X) \;, $$
  holds.
 \end{itemize}
\end{thm}

\medskip

Symplectic structures and complex structures provide the fundamental examples of generalized complex structures; in fact, the following generalized Darboux theorem by M. Gualtieri holds. (Recall that a \emph{regular point} of a generalized complex manifold is a point at which the type of the generalized complex structure is locally constant.)

\begin{thm}[{\cite[Theorem 4.35]{gualtieri-phdthesis}, \cite[Theorem 3.6]{gualtieri}}]
 For any regular point of a $2n$-dimensional generalized complex manifold with type equal to $k$, there is an open neighbourhood endowed with a set of local coordinates such that the generalized complex structure is a $B$-field transform of the standard generalized complex structure of $\C^{k}\times\R^{2n-2k}$.
\end{thm}

The standard generalized complex structure of constant type $n$ (that is, locally equivalent to the standard complex structure of $\C^n$), the generalized complex structure of constant type $0$ (that is, locally equivalent to the standard symplectic structure of $\R^{2n}$), and the $B$-field transform of a generalized complex structure are recalled in the following examples. See also \cite[Example 4.12]{gualtieri-phdthesis}.

\begin{ex}[{Generalized complex structures of type $n$, \cite[Example 4.11, Example 4.25]{gualtieri-phdthesis}}] Let $X$ be a compact $2n$-dimensional manifold endowed with a complex structure $J\in\End(TX)$. Consider the ($0$-twisted) generalized complex structure
$$ \mathcal{J}_J
\;:=\;
\left(
\begin{array}{c|c}
 -J & 0 \\
\hline
 0 & J^*
\end{array}
\right)
\;\in\; \End\left(TX\oplus T^*X\right)
\;,$$
where $J^*\in\End(T^*X)$ denotes the dual endomorphism of $J\in\End(TX)$. Note that the $\im$-eigenspace of the $\C$-linear extension of $\mathcal{J}_J$ to $\left(TX \oplus T^*X \right)\otimes_\C \C$ is
$$ L_{\mathcal{J}_J} \;=\; T^{0,1}_JX \oplus \left(T^{1,0}_JX\right)^* \;,$$
and the canonical bundle is
$$ U^n_{\mathcal{J}_J} \;=\; \wedge^{n,0}_JX \;.$$

Hence, one gets that, \cite[Example 4.25]{gualtieri-phdthesis},
$$ U^\bullet_{\mathcal{J}_J} \;=\; \bigoplus_{p-q=\bullet}\wedge^{p,q}_JX \;,$$
and that
$$ \del_{\mathcal{J}_J} \;=\; \del_J \qquad \text{ and } \qquad \delbar_{\mathcal{J}_J} \;=\; \delbar_J \;; $$
note that $\de^{\mathcal{J}_J}$ is the operator $\de^c_J:=-\im(\del-\delbar)$, \cite[Remark 4.26]{gualtieri-phdthesis}.
Note also that $X$ satisfies the $\de\de^{\mathcal{J}_J}$-Lemma if and only if $X$ satisfies the $\de\de^c_J$-Lemma, and that the Hodge and Fr\"olicher spectral sequence associated to the canonical double complex $\left(\Doub^{\bullet,\bullet}U^\bullet_{\mathcal{J}_J},\, \del_{\mathcal{J}_J} \otimes_\R \id,\, \delbar_{\mathcal{J}_J} \otimes_\R\beta\right)$ degenerates at the first level if and only if the Hodge and Fr\"olicher spectral sequence associated to the double complex $\left(\wedge^{\bullet,\bullet}_JX,\, \del_J,\, \delbar_J\right)$ does, \cite[Remark at page 76]{cavalcanti}.

In particular, it follows that, for $\sharp\in\left\{\delbar,\, \del,\, BC,\, A\right\}$,
$$ GH_{\sharp_{\mathcal{J}_J}}^\bullet(X) \;=\; \Tot^\bullet H_{\sharp_J}^{\bullet,-\bullet}(X) \;=\; \bigoplus_{p-q=\bullet} H_{\sharp_J}^{p,q}(X) \;. $$

Therefore, by Theorem \ref{thm:gen-frol-ineq} and Theorem \ref{thm:gen-charact}, and by using the equalities $\dim_\C H^{\bullet_1,\bullet_2}_{{BC}_J}(X) = \dim_\C H^{n-\bullet_2,n-\bullet_1}_{A_J}(X)$ and $\dim_\C H^{\bullet_1,\bullet_2}_{\delbar_J}(X) = \dim_\C H^{n-\bullet_2,n-\bullet_1}_{\del_J}(X)$, one gets the following result, compare Corollary \ref{cor:cplx}, \cite[Theorem A, Theorem B]{angella-tomassini-3}.

\begin{cor}
 Let $X$ be a compact complex manifold. Then the inequality
 $$ \sum_{p-q=\bullet} \dim_\C H_{{BC}_J}^{p,q}(X) \;\geq\; \sum_{p-q=\bullet} \dim_\C H^{p,q}_{\delbar_J}(X) $$
 holds. Furthermore, $X$ satisfies the $\del_J\delbar_J$-Lemma if and only if
 \begin{inparaenum}[\itshape (i)]
  \item the Hodge and Fr\"olicher spectral sequence of $X$ degenerates at the first level, namely,
  $$ \dim_\C H^\bullet_{dR}(X;\C) \;=\; \dim_\C \Tot^\bullet H^{\bullet,\bullet}_{\delbar_J}(X) \;,$$
  and
  \item the equality
  $$ \sum_{p-q=\bullet} \dim_\C H^{p,q}_{{BC}_J}(X) \;=\; \sum_{p-q=\bullet} \dim_\C H^{p,q}_{\delbar_J}(X) $$
  holds.
 \end{inparaenum}
\end{cor}

\begin{landscape}
\smallskip
\begin{footnotesize}
\begin{center}
\begin{table}
\begin{tabular}{c||*{4}{c}|*{4}{c}|*{4}{c}|*{4}{c}|*{4}{c}|*{4}{c}|*{4}{c}||}
\toprule
   $\mathbb{I}_3:=\left.\Z[\im]^3\middle\backslash\C^3\right.$ & \multicolumn{4}{c|}{$\dim_\C \Tot^{-3} H^{\bullet,-\bullet}_{\sharp}(X)$} & \multicolumn{4}{c|}{$\dim_\C \Tot^{-2} H^{\bullet,-\bullet}_{\sharp}(X)$} & \multicolumn{4}{c|}{$\dim_\C \Tot^{-1} H^{\bullet,-\bullet}_{\sharp}(X)$} & \multicolumn{4}{c|}{$\dim_\C \Tot^{0} H^{\bullet,-\bullet}_{\sharp}(X)$} & \multicolumn{4}{c|}{$\dim_\C \Tot^{1} H^{\bullet,-\bullet}_{\sharp}(X)$} & \multicolumn{4}{c|}{$\dim_\C \Tot^{2} H^{\bullet,-\bullet}_{\sharp}(X)$} & \multicolumn{4}{c||}{$\dim_\C \Tot^{3} H^{\bullet,-\bullet}_{\sharp}(X)$}\\
  {\bfseries classes} & $\delbar$ & $\del$ & $BC$ & $A$ & $\delbar$ & $\del$ & $BC$ & $A$ & $\delbar$ & $\del$ & $BC$ & $A$ & $\delbar$ & $\del$ & $BC$ & $A$ & $\delbar$ & $\del$ & $BC$ & $A$ & $\delbar$ & $\del$ & $BC$ & $A$ & $\delbar$ & $\del$ & $BC$ & $A$\\
\midrule[0.02em]\midrule[0.02em]
{\itshape (i)} & 1 & 1 & 1 & 1 & 5 & 5 & 5 & 5 & 11 & 11 & 11 & 11 & 12 & 12 & 12 & 12 & 11 & 11 & 11 & 11 & 5 & 5 & 5 & 5 & 1 & 1 & 1 & 1 \\
\midrule[0.02em]
{\itshape (ii.a)} & 1 & 1 & 1 & 1 & 4 & 4 & 4 & 4 & 9 & 9 & 11 & 11 & 10 & 10 & 11 & 11 & 9 & 9 & 11 & 11 & 4 & 4 & 4 & 4 & 1 & 1 & 1 & 1 \\
{\itshape (ii.b)} & 1 & 1 & 1 & 1 & 4 & 4 & 4 & 4 & 9 & 9 & 11 & 11 & 10 & 10 & 10 & 10 & 9 & 9 & 11 & 11 & 4 & 4 & 4 & 4 & 1 & 1 & 1 & 1 \\
\midrule[0.02em]
{\itshape (iii.a)} & 1 & 1 & 1 & 1 & 3 & 3 & 3 & 3 & 8 & 8 & 11 & 11 & 10 & 10 & 11 & 11 & 8 & 8 & 11 & 11 & 3 & 3 & 3 & 3 & 1 & 1 & 1 & 1 \\
{\itshape (iii.b)} & 1 & 1 & 1 & 1 & 3 & 3 & 3 & 3 & 8 & 8 & 11 & 11 & 10 & 10 & 10 & 10 & 8 & 8 & 11 & 11 & 3 & 3 & 3 & 3 & 1 & 1 & 1 & 1 \\
\midrule[0.02em]\midrule[0.02em]
 & \multicolumn{4}{c|}{$\mathbf{b_0=1}$} & \multicolumn{4}{c|}{$\mathbf{b_1=4}$} & \multicolumn{4}{c|}{$\mathbf{b_2=8}$} & \multicolumn{4}{c|}{$\mathbf{b_3=10}$} & \multicolumn{4}{c|}{$\mathbf{b_4=8}$} & \multicolumn{4}{c|}{$\mathbf{b_5=4}$} & \multicolumn{4}{c||}{$\mathbf{b_6=1}$} \\
\bottomrule
\end{tabular}
 \caption{Generalized complex cohomologies of the Iwasawa manifold.}
 \label{table:iwasawa}
\end{table}
\end{center}
\end{footnotesize}
\smallskip
\end{landscape}

\end{ex}

\begin{ex}[{Generalized complex structures of type $0$, \cite[Example 4.10]{gualtieri-phdthesis}}] Let $X$ be a compact $2n$-dimensional manifold endowed with a symplectic structure $\omega \in \wedge^2 X \simeq \Hom\left(TX; T^*X\right)$.
Consider the ($0$-twisted) generalized complex structure
$$ \mathcal{J}_\omega \;:=\;
\left(
\begin{array}{c|c}
 0 & -\omega^{-1} \\
\hline
 \omega & 0
\end{array}
\right) \;,$$
where $\omega^{-1}\in\Hom\left(T^*X; TX\right)$ denotes the inverse of $\omega \in \Hom\left(TX; T^*X\right)$. Note that the $\im$-eigenspace of the $\C$-linear extension of $\mathcal{J}_\omega$ to $\left(TX \otimes_\R \C\right) \oplus \left(T^*X \otimes_\R \C\right)$ is
$$ L_{\mathcal{J}_\omega} \;=\; \left\{ X -\im \, \omega\left(X,\,\sspace\right) \st X \in TX\otimes_\R\C \right\} \;, $$
which has Clifford annihilator $\exp(\im\,\omega)$,
and the canonical bundle is
$$ U_{\mathcal{J}_\omega}^n \;=\; \C\left\langle \exp\left(\im\,\omega\right) \right\rangle \;.$$
In particular, one gets that, \cite[Theorem 2.2]{cavalcanti-jgp},
$$ U_{\mathcal{J}_\omega}^{n-\bullet} \;=\; \exp{\left(\im\omega\right)}\, \left(\exp{\left(\frac{\Lambda}{2\im}\right)} \left(\wedge^\bullet X \otimes_\R \C\right)\right) \;, $$
where $\Lambda := -\iota_{\omega^{-1}}$.
Note that, \cite[\S2.2]{cavalcanti-jgp},
$$ \de^\mathcal{J_\omega} \;=\; \de^\Lambda \;. $$

By considering the natural isomorphism
$$ \varphi\colon \wedge^\bullet X \otimes_\R \C \to \wedge^\bullet X \otimes_\R \C \;, \qquad \varphi(\alpha) \;:=\; \exp{\left(\im\omega\right)}\, \left(\exp{\left(\frac{\Lambda}{2\im}\right)}\, \alpha\right) \;,$$
one gets that, \cite[Corollary 1]{cavalcanti-jgp},
$$ \varphi\left(\wedge^\bullet X\otimes_\R\C\right) \simeq U^{n-\bullet} \;, \qquad \text{ and }\qquad \varphi\de \;=\; \delbar_{\mathcal{J}_\omega}\varphi  \quad\text{ and }\quad \varphi\de^{\mathcal{J}_\omega} \;=\; -2\im\del_{\mathcal{J}_\omega}\varphi \;; $$
in particular,
$$ GH^\bullet_{\delbar_{\mathcal{J}_\omega}}(X) \;\simeq\; H^{n-\bullet}_{dR}(X;\C) \;. $$

In particular, one recovers Theorem \ref{thm:sympl}, namely,
$$ \dim_\R H^{\bullet}_{\left( \de , \de^\Lambda ; \de\de^\Lambda \right)}\left(X\right) + \dim_\R H^{\bullet}_{\left( \de\de^\Lambda ; \de , \de^\Lambda \right)}\left(X\right) \;\geq\; 2\, \dim_\R H^{\bullet}_{dR}(X;\R) \;, $$
and the equality holds if and only if $X$ satisfies the Hard Lefschetz Condition.
\end{ex}

\begin{ex}[{$B$-transform, \cite[\S3.3]{gualtieri-phdthesis}}] Let $X$ be a compact $2n$-dimensional manifold endowed with an $H$-twisted generalized complex structure $\mathcal{J}$, and let $B$ be a $\de$-closed $2$-form. Consider the $H$-twisted generalized complex structure
$$ \mathcal{J}^B \;:=\; \exp \left(-B\right) \, \mathcal{J} \, \exp B \qquad \text{ where } \qquad \exp B \;=\;
\left(
\begin{array}{c|c}
 \id_{TX} & 0 \\
\hline
 B & \id_{T^*X}
\end{array}
\right) \;.$$
Note that the $\im$-eigenspace of the $\C$-linear extension of $\mathcal{J}$ to $\left(TX \oplus T^*X \right)\otimes_\R \C$ is, \cite[Example 2.3]{cavalcanti},
$$ L_{\mathcal{J^B}} \;=\; \left\{X+\xi-\iota_XB \st X+\xi \in L_{\mathcal{J}}\right\} \;, $$
and the canonical bundle is, \cite[Example 2.6]{cavalcanti},
$$ U^n_{\mathcal{J}^B} \;=\; \exp B \wedge U^n_{\mathcal{J}} \;.$$

Hence one gets that, \cite[\S2.3]{cavalcanti-jgp},
$$ U^{\bullet}_{\mathcal{J}^B} \;=\; \exp B \wedge U^\bullet_{\mathcal{J}} \;. $$
and that, \cite[\S2.3]{cavalcanti-jgp},
$$ \del_{\mathcal{J}^B} \;=\; \exp \left(-B\right) \, \del_{\mathcal{J}} \, \exp B \qquad \text{ and } \qquad \delbar_{\mathcal{J}^B} \;=\; \exp \left(-B\right) \, \delbar_{\mathcal{J}} \, \exp B \;. $$

In particular, one gets that $\mathcal{J}$ satisfies the $\del_{\mathcal{J}}\delbar_\mathcal{J}$-Lemma if and only if $\mathcal{J}^B$ satisfies the $\del_{\mathcal{J}^B}\delbar_{\mathcal{J}^B}$-Lemma.
\end{ex}

\begin{rem}
 We recall that, given a $\de$-closed $3$-form $H$ on a manifold $X$, an \emph{$H$-twisted generalized K\"ahler structure} on $X$ is a pair $\left( \mathcal{J}_1,\, \mathcal{J}_2\right)$ of $H$-twisted generalized complex structures on $X$ such that
 \begin{inparaenum}[\itshape (i)]
  \item $\mathcal{J}_1$ and $\mathcal{J}_2$ commute, and
  \item the symmetric pairing $\left\langle \mathcal{J}_1 \sspace,\, \mathcal{J}_2\ssspace\right\rangle$ is positive definite.
 \end{inparaenum}
 Generalized K\"ahler geometry is equivalent to a bi-Hermitian geometry with torsion, \cite[Theorem 2.18]{gualtieri-kahler}.

 We recall that a compact manifold $X$ endowed with an $H$-twisted generalized K\"ahler structure $\left( \mathcal{J}_1,\, \mathcal{J}_2\right)$ satisfies both the $\de_H\de_H^{\mathcal{J}_1}$-Lemma and the $\de_H\de_H^{\mathcal{J}_2}$-Lemma, \cite[Corollary 4.2]{gualtieri-kahler}.

 Any K\"ahler structure provide an example of a $0$-twisted generalized K\"ahler structure.
 A left-invariant non-trivial twisted generalized K\"ahler structure on a (non-completely solvable) solvmanifold (which is the total space of a $\mathbb{T}^2$-bundle over the Inoue surface, \cite[Proposition 3.2]{fino-tomassini-jsg}) has been constructed by A. Fino and the second author, \cite[Theorem 3.5]{fino-tomassini-jsg}.
\end{rem}

\begin{rem}
 Note that A. Tomasiello proved in \cite[\S B]{tomasiello} that satisfying the $\de\de^\mathcal{J}$-Lemma is a stable property under small deformations.
\end{rem}

\end{document}